\documentclass[12pt,letterpaper,reqno]{amsart}
\usepackage{fullpage}
\usepackage{amsmath,amsthm,amssymb,amscd}
\usepackage{enumerate}
\usepackage{enumitem}
\usepackage{array}
\usepackage{float}
\usepackage{bbm}
\usepackage{bm}
\usepackage{stmaryrd}
\usepackage{mathrsfs}
\usepackage{comment}
\usepackage{mathtools}

\usepackage{hyperref}
\hypersetup{colorlinks=true,linkcolor=blue,citecolor=blue,urlcolor=blue}

\theoremstyle{plain}

\newtheorem*{unnumberedproposition}{Proposition}
\newtheorem*{unnumberedlemma}{Lemma}

\newtheorem{theorem}{Theorem}[section]
\newtheorem{proposition}[theorem]{Proposition}
\newtheorem{lemma}[theorem]{Lemma}

\theoremstyle{definition}
\newtheorem{remark}[theorem]{Remark}

\mathtoolsset{showonlyrefs=true}
\numberwithin{equation}{section}
\numberwithin{figure}{section}
\numberwithin{table}{section}

\let\subsectiontemp\subsection
\renewcommand{\subsection}[1]{ 
    \subsectiontemp{#1} \hfill\vspace{0.5\linespacing} 
}

\allowdisplaybreaks[1]

\newcommand{\lrabs}[1]{\!\left\lvert #1 \right\lvert}
\newcommand{\lrp}[1]{\!\left(#1\right)}
\newcommand{\lrb}[1]{\!\left[#1\right]}
\newcommand{\lrcb}[1]{\!\left\{#1\right\}}

\newcommand{\M}{\mathcal{M}}

\newcommand{\EE}{\mathcal{E}}

\newcommand{\new}{{\rm new}}
\newcommand{\nw}{{\rm new}}

\author[T. Nelson]{Timothy Nelson}
\address[T. Nelson]{School of Science, Technology and Health, Gordon College, Wenham, MA}

\author[E. Ross]{Erick Ross}
\address[E. Ross]{School of Mathematical and Statistical Sciences, Clemson University, Clemson, SC}
\email{erickjohnross@gmail.com}

\author[M. Wassercug]{Maya Wassercug}
\address[M. Wassercug]{Department of Mathematics, Northwestern University, Evanston, IL}

\author[H. Xue]{Hui Xue}
\address[H. Xue]{School of Mathematical and Statistical Sciences, Clemson University, Clemson, SC}
\email{huixue@clemson.edu}



\title{On Murmurations}

\begin{document}

\begin{abstract}
    In the recent paper ``Murmurations", Zubrilina showed that the Fourier coefficients of modular forms follow a certain murmuration pattern. The proof given in ``Murmurations" proceeded by estimating the terms of the Skoruppa-Zagier trace formula. However, a few technical issues arise in the proof, including an omission of certain Chebyshev factors that arise in summations from the Skoruppa-Zagier trace formula. In this work, we address these issues to solidify the proof.
\end{abstract}

\maketitle

\tableofcontents

%%%%%%%%%%%%%%%%%%%%%%%%%%%%%%%%%%%%%%%%%%%%%%%%%%%%%%
%%%%%%%%%%%%%%%%%%%% INTRODUCTION %%%%%%%%%%%%%%%%%%%%
%%%%%%%%%%%%%%%%%%%%%%%%%%%%%%%%%%%%%%%%%%%%%%%%%%%%%%

\section{Introduction}
For even integers $k \ge 2$ and integers $N \ge 1$, let $S_k^\new(N)$ denote the newspace of cusp forms of weight $k$ and level $N$. Also, let $H^\new (N ,k)$ denote the normalized newform basis for $S_k^\new(N)$. For newforms $f \in H^\new(N,k)$, let $\varepsilon(f)\in \{\pm 1\}$ denote the root number of $f$, $a_f(n)$ denote the $n$-th Fourier coefficient of $f$, and $\lambda_f(n) := n^{-\frac{k-1}{2}}  a_f(n)$. 
Also, let $U_{n}$ denote the $n$-th Chebyshev polynomial of the second kind $U_n(\cos (\theta)) := \frac{\sin((n+1)\theta)}{\sin (\theta)}$, and $H_1(D) := H(|D|)$ for $D \le 0$, where $H(\cdot)$ denotes the Hurwitz class number.
Additionally, we let $\sum^\square$ denote a summation over squarefree parameters.
Moreover, we let
\begin{align}
    \nu(r) &:= \prod_{p\mid r}\lrp{1+\frac{p^2}{p^4-2p^2-p+1}}.
\end{align}

Recently, He-Lee-Oliver-Pozdnyakov \cite{He-Lee-Oliver-Pozdnyakov} discovered a pattern relating to the trace of Frobenius for elliptic curves, which they called ``murmurations". In particular, the average of these traces of Frobenius over elliptic curves with fixed rank tends to a continuous, oscillating function. 

This pattern was also observed in modular forms, namely for the average of $a_f(P)$ over weight $k$ newforms with fixed root number and level within a certain range. Specifically, Zubrilina \cite{Zubrilina} showed the following result:
\begin{theorem}[{\cite[Theorem 1]{Zubrilina}}] \label{thm:main-result} 
Fix an even integer $k\ge 2$. Let $P$, $X$, and $Y$ be parameters going to infinity with $X,Y\in \mathbb{R}_{+}$ and $P$ prime; assume further that $Y=(1+o(1))X^{1-\delta_2},$ and $P\ll X^{1+\delta_1}$ for some $\delta_1,\delta_2$ with $0<\delta_1<\frac{1}{11}$, $2\delta_1<\delta_2<\frac{1}{13} (4-18\delta_1).$ Let $y := \frac{P}{X}$. Then 
\begin{align}
    &\frac{
    \sum_{N\in[X,X+Y]}^\square \sum_{f\in H^{\rm new}(N,k)}\sqrt{P}\lambda_f(P)\varepsilon(f)
    }{
    \sum_{N\in[X,X+Y]}^\square \sum_{f\in H^{\rm new}(N,k)}1 
    }  \\
    &=\frac{\alpha(-1)^{\frac{k}{2}-1}}{k-1}\sum_{1\leq r\leq 2\sqrt{y}}U_{k-2}\lrp{\frac{r}{2\sqrt{y}}}\nu(r)\sqrt{4y-r^2} +\frac{\beta}{k-1}\sqrt{y}-\gamma \delta_{k=2}y  \\
    &\quad\ +O_\varepsilon\lrp{X^{-\delta'+\varepsilon}+\frac{1}{P}}
\end{align}
where $\delta'>0$ is explicitly expressible through $\delta_1,\delta_2,$ and 
\begin{align}
\alpha= 2\pi \prod_{p}\frac{1-p-2p^2+p^4}{p^4-2p^2+p}, \quad \beta=2\pi\prod_{p}\frac{-1+p^2+p^3}{p(-1+p+p^2)}, \quad 
\gamma=12\prod_{p}\frac{p(1+p)}{-1+p+p^2}.
\end{align}
\end{theorem}

The proof of Theorem \ref{thm:main-result} given in \cite{Zubrilina} proceeded by estimating the Skoruppa-Zagier trace formula, which \cite[Section 2]{Zubrilina} stated as the following: for squarefree integers $N > 1$ and primes $P\nmid N$,  
\begin{align}
    &\sum_{f \in H^\new(N,k)}\sqrt{P}\lambda_f(P)\varepsilon(f) \\
    &= \frac{H_1(-4PN)}{2} + (-1)^{\frac{k}{2}-1}U_{k-2}\lrp{\frac{r\sqrt{N}}{2\sqrt{P}}}\sum_{1 \le r \le 2\sqrt{\frac{P}{N}}}H_1(r^2N^2-4PN) - \delta_{k=2}(P+1). \quad
    \label{eqn:incorrect-Skoruppa-Zagier-trace-formula}
\end{align}

However, one can see that in this statement of the trace formula,
the Chebyshev polynomial term $U_{k-2}\lrp{\frac{r \sqrt N}{2 \sqrt P}}$ appears outside of the summation over $r$, and yet depends on the summation variable $r$. The subsequent proof of Theorem \ref{thm:main-result} given in \cite{Zubrilina} utilizes this malformed statement of the trace formula and only estimates the above summation without the $U_{k-2}\lrp{\frac{r \sqrt N}{2 \sqrt P}}$ term. 

To address this issue, the first main goal of this paper is to provide an alternate proof of Theorem \ref{thm:main-result} based on a corrected version of the Skoruppa-Zagier trace formula (Lemma \ref{lem:skoruppa-zagier-trace-formula}). One of the key technical inputs needed is Propostion \ref{prop:main-proposition}, which constitutes a weighted version of \cite[Proposition 3.2]{Zubrilina} (i.e. including the weight $U_{k-2}\lrp{\frac{r \sqrt N}{2 \sqrt P}}$). 

We remark that the Zubrilina's work \cite{Zubrilina} provides both the original proof strategy and the
main analytic framework for Theorem \ref{thm:main-result}; our purpose here is just to clarify several technical points in the argument and to supply alternative proofs where needed.

After proving Theorem \ref{thm:main-result}, \cite[Theorem 6]{Zubrilina} later shows the following conclusion, conditionally on the Riemann Hypothesis. We state the conclusion here unconditionally as Theorem \ref{thm:second-main-result}. Note that the definition of $\M_k(y)$ is given in \eqref{eq:Mk-bessel}.
\begin{theorem} \label{thm:second-main-result}
    For $c>1$ and $k\ge6$ even, define the smoothing of $\M_k(y)$ as 
    \begin{align}
        \M_k^{c}(y)
        :=
        \frac{\displaystyle\int_1^c \M_k\lrp{\frac{y}{u}}u\, du}
        {\displaystyle\int_1^c u\, d u}.
    \end{align}
    Then $\M_k^{c}$ is continuous on $[0,\infty)$ with $\M_k^{c}(0)=0$ and as $y\to\infty$
    \begin{align} \label{eqn:second-main-thm__intro-statement__error-bound}
        \M_k^{c}(y)=\frac{1}{2}+o_{k,c}(1).
    \end{align}
\end{theorem}

The proof of this theorem given in \cite{Zubrilina} starts by showing in \cite[Lemma 7.3]{Zubrilina} that (assuming the Riemann Hypothesis) a certain function $R(T)$ is bounded by $R(T) \ll_\varepsilon  T^{-2+\varepsilon}$. Then the proof of Theorem \ref{thm:second-main-result} given in \cite{Zubrilina} depends heavily on this exponent $-2+\varepsilon$.

However, this proof strategy seems to not quite work because the bound $R(T) \ll_\varepsilon  T^{-2+\varepsilon}$ given in \cite[Lemma 7.3]{Zubrilina} is not amply justified.
Let 
\begin{align}
    L(s):=\frac{\zeta(s+2)}{\zeta(2s+4)}\prod_{p}\left(1+\frac{-1+p+2p^2}{(1-p-2p^2+p^4)(1+p^{2+s})}\right).
\end{align}
Then the proof of \cite[Lemma 7.3]{Zubrilina} proceeds by attempting to shift the integral of $L(s)\frac{T^s}{s}$ from the contour given by $\mathrm{Re}(s)=1$ to the contour given by $\mathrm{Re}(s)=-2+\varepsilon$. The justification given for this shift was that $L(s)\frac{T^s}{s}$ only has poles at $s=0,-1$ in the vertical strip $-2+\varepsilon \le \mathrm{Re}(s) \le 1$. However, it seems this contour shifting argument does not go through because $\zeta(2s+4)$ also has zeros on the vertical line $\mathrm{Re}(s)=-\frac{7}{4}$.

Because of this issue, the second main goal of this paper is to give an alternative proof of Theorem \ref{thm:second-main-result}, using a different method. In fact, the proof we give is unconditional; we do not need to assume the Riemann Hypothesis. Additionally, we will obtain power-savings in the error term of \eqref{eqn:second-main-thm__intro-statement__error-bound}, namely $O_{k,c}\big(y^{\frac{-1}{12}}\big)$. See the statement of Theorem \ref{thm:second-main-result} given in Section \ref{sec:proof-of-second-main-result} for the precise statement of our result.

In closing, we would like to point out that the results we prove in this paper are also necessary to extend the work of Kundu and M\"uller in \cite{kundu-muller}. This paper proved a murmuration result for $k=2$, and then suggested an extension to general weight $k$ \cite[Subsection 1.1]{kundu-muller}. However, the main obstacle to this extension was that the arguments of \cite{kundu-muller} only deal with unweighted summations (i.e. without $U_{k-2} \lrp{\frac{r\sqrt{N}}{2\sqrt P}}$ inside the Skoruppa-Zagier trace formula summation). The arguments of the first half of the paper (i.e. dealing with weighted summations) are precisely the ones needed to overcome this obstacle. 
% Note that this obstacle was not relevant in the setting of \cite{kundu-muller} at $k=2$ because in this case, $U_{k-2} \lrp{\frac{r\sqrt{N}}{2\sqrt P}}$ is identically $1$.

Finally, we give an outline of the paper. Section \ref{sec:proof-of-main-prop} proves the weighted class number estimate Proposition \ref{prop:main-proposition} (which replaces \cite[Proposition 3.2]{Zubrilina}). Then Section \ref{sec:proof-of-main-thm} uses Proposition \ref{prop:main-proposition} to prove Theorem \ref{thm:main-result}. Next, Section \ref{sec:proof-of-prop-S-estimate} proves Proposition \ref{prop:S-estimate}. Lastly, Section \ref{sec:proof-of-second-main-result} uses Proposition \ref{prop:S-estimate} to prove Theorem \ref{thm:second-main-result}.

%%%%%%%%%%%%%%%%%%%%%%%%%%%%%%%%%%%%%%%%%%%%%%%%%%%%%%
%%%%%%%%%%%%% PROVING THE MODIFIED PROPOSITION %%%%%%%  
%%%%%%%%%%%%%%%%%%%%%%%%%%%%%%%%%%%%%%%%%%%%%%%%%%%%%%

\section{Proof of Proposition \ref{prop:main-proposition}} \label{sec:proof-of-main-prop}

Throughout Sections \ref{sec:proof-of-main-prop} and \ref{sec:proof-of-main-thm}, we will use the two constants
\begin{align}
    A_0 &:= \frac{1}{\zeta(2)\pi} \prod_{p}\frac{p^3+p^2-1}{(p+1)^2(p-1)}, 
    \label{eqn:def-A0}
    \\
    B_0 &:= \frac{1}{\zeta(2)\pi} \prod_{p} \frac{p^4-2p^2-p+1}{(p^2-1)^2}.
\end{align}
Note that \cite{Zubrilina} instead used the constants $A = \zeta(2)\pi A_0$ and $B = \zeta(2)\pi B_0$.
Also, since our results are asymptotic in $X$, we will assume throughout that $X,Y$ are sufficiently large (so that when $Y = o(X)$, for example, we can utilize the bound $Y < X$).

The goal of this section is to prove Proposition \ref{prop:main-proposition}, which will replace \cite[Proposition 3.2]{Zubrilina} in our proof of Theorem \ref{thm:main-result}. For reference, we first write down \cite[Proposition 3.2]{Zubrilina}, as stated in \cite{Zubrilina}.

\begin{unnumberedproposition}[\cite{Zubrilina}, Proposition 3.2]
Let $P>2$ be prime and let $[X,X+Y]$ be an interval of length $Y=o(X)$. Let $y:= \frac{P}{X}$. Then as $X\to\infty$,
\begin{align}
       &\frac{1}{XY} 
       \sum_{1\leq r \leq 2\sqrt{\frac{P}{X}}}
       \hspace{2mm}
       \sideset{}{^\square}\sum_{\substack{N\in[X,X+Y] \\ P\nmid N}}
        H_1(r^2N^2-4PN)   \\
        &= \sum_{1\leq r\leq 2\sqrt{y}}B_0\nu(r)\sqrt{4y-r^2} +O_\varepsilon
        \lrp{
        \lrb{
            \frac{P^{\frac{11}{10}}}{X^{\frac{9}{10}}Y^{\frac{2}{5}}}
            +\frac{PY}{X^2} 
            +\frac{PY^{\frac{1}{2}}}{X^{\frac{3}{2}}}
            +\frac{P}{X^{\frac{1}{2}}Y^{\frac{13}{18}}}
            +\frac{P}{XY^{\frac{1}{9}}}
        }
        \!\lrp{PXY}^\varepsilon
        \!}\!.
\end{align}
\end{unnumberedproposition}

In Proposition \ref{prop:main-proposition}, we will modify this result to 
also account for the factor of $U_{k-2}\lrp{\frac{r\sqrt{N}}{2\sqrt{P}}}$ that was omitted in \cite{Zubrilina}.
For technical reasons, we will also swap the summations to sum over $1 \le r \le 2\sqrt{\frac{P}{N}}$ rather than $1 \le r \le 2\sqrt{\frac{P}{X}}$. 
\begin{proposition} \label{prop:main-proposition}
    Fix an even integer $k\ge2$.
    Let $P>2$ prime and $[X,X+Y]$ be an interval of length $Y=o(X)$. Let $y:= \frac{P}{X}$ and
    \begin{align}  
        \EE = \EE(P,X,Y,\varepsilon) :=
        \lrb{
            1 
            + 
            \frac{
                P^\frac{1}{10} X^\frac{3}{5} 
            }{
                Y^\frac{9}{10}
            }
            +
            \frac{X}{Y^\frac{11}{9}}
            +
            \frac{X^\frac{1}{2}}{Y^\frac{11}{18}}
        }
        (PX)^\varepsilon.
    \end{align}
    Then 
    \begin{align}\label{eq:main-prop-LHS}
       &\frac{1}{XY} 
       \sideset{}{^\square}\sum_{\substack{N\in[X,X+Y] \\ P\nmid N}}
       \hspace{2mm}
       \sum_{1\leq r \leq 2\sqrt{\frac{P}{N}}} U_{k-2}\lrp{\frac{r\sqrt{N}}{2\sqrt{P}}} H_1(r^2N^2-4PN) \\
       &= B_0\sum_{1\leq r\leq 2\sqrt{y}}U_{k-2}\lrp{\frac{r}{2\sqrt{y}}}\nu(r)\sqrt{4y-r^2} 
       +
       O_\varepsilon\lrp{ k
        \lrb{
            \frac{P^\frac{1}{2}}{X^\frac{1}{2}}
            + 1
        }
        \frac{P^\frac{1}{2} Y^\frac{1}{2}}{X}
        \EE
       }.
       \label{eqn:main-prop-error-term}
    \end{align}
\end{proposition}
\begin{proof}
    Observe that 
    \begin{align}
        & \frac{1}{XY} \sideset{}{^\square}\sum_{\substack{N\in[X,X+Y] \\ P\nmid N}}\sum_{1\leq r \leq 2\sqrt{\frac{P}{N}}} U_{k-2}\lrp{\frac{r\sqrt{N}}{2\sqrt{P}}} H_1(r^2N^2-4PN) \\
        %%%%%%%%%%%%%%%%%%%%%%%%%%%%%%%%%%%%%%%%%%
        &=\frac{1}{XY}
        \sum_{1\leq r< 2\sqrt{\frac{P}{X+Y}}}\sideset{}{^\square} \sum_{\substack{N\in [X,X+Y] \\ P \nmid N}}
        U_{k-2}\lrp{\frac{r\sqrt{N}}{2\sqrt{P}}} H_1(r^2N^2-4PN)\\
        %%%%%%%%%%%%%%%%%%%%%%%%
        &  \quad + \frac{1}{XY}
        \sideset{}{^\square}\sum_{\substack{N\in[X,X+Y] \\ P\nmid N}}
        \sum_{2\sqrt{\frac{P}{X+Y}}\le r \le 2\sqrt{\frac{P}{N}}}  U_{k-2}\lrp{\frac{r\sqrt{N}}{2\sqrt{P}}} H_1(r^2N^2-4PN). \\
        %%%%%%%%%%%%%%%%%%%%%%%%%%%%%%%%%%%%%%%%%%
        &= \frac{1}{XY}\sum_{1\leq r< 2\sqrt{\frac{P{}}{X+Y}}} U_{k-2}\lrp{\frac{r\sqrt{X}}{2\sqrt{P}}}
        \sideset{}{^\square}\sum_{\substack{N\in[X,X+Y]\\ P\nmid N}} H_1(r^2N^2-4PN) \label{eq:initial-double-sum-main} \\
        %%%%%%%%%%%%%%%%%%%%%%%%%
        &\quad + \frac{1}{XY}\sum_{1\leq r< 2\sqrt{\frac{P{}}{X+Y}}}\,\, \sideset{}{^\square }\sum_{\substack{N\in[X,X+Y]\\ P\nmid N}} E(r,N)H_1(r^2N^2-4PN)  \label{eq:initial-double-sum-first-error} \\
        %%%%%%%%%%%%%%%%%%%%%%%%%
        &\quad + \frac{1}{XY}
        \sideset{}{^\square}\sum_{\substack{N\in[X,X+Y] \\ P\nmid N}}
        \sum_{2\sqrt{\frac{P}{X+Y}}\le r \le 2\sqrt{\frac{P}{N}}}  U_{k-2}\lrp{\frac{r\sqrt{N}}{2\sqrt{P}}} H_1(r^2N^2-4PN),
        \label{eq:initial-double-sum-second-error}
    \end{align}
    where 
    \begin{align} \label{eqn:def-E(r,N)}
        E(r,N) := U_{k-2}\lrp{\frac{r}{2}\frac{\sqrt{N}}{\sqrt{P}}}-
        U_{k-2}\lrp{\frac{r}{2}\frac{\sqrt{X}}{\sqrt{P}}}.
    \end{align}
    We show that \eqref{eq:initial-double-sum-main} constitutes the main term of the desired result (up to the claimed error term), and that
    \eqref{eq:initial-double-sum-first-error}
    and
    \eqref{eq:initial-double-sum-second-error}
    are bounded by the claimed error term.

    First, we have that
    \begin{align}
        \eqref{eq:initial-double-sum-main} 
        &= \frac{1}{XY}\sum_{1\leq r<2\sqrt{\frac{P{}}{X+Y}}} U_{k-2}\lrp{\frac{r\sqrt{X}}{2\sqrt{P}}}
        \sideset{}{^\square}\sum_{\substack{N\in[X,X+Y]\\ P\nmid N}} H_1(r^2N^2-4PN) \\
        %%%%%%%%%%%%%%%%%%%%%%%%%%%%%%%%%%%%%%%%
        &= \frac{1}{XY} \sum_{1\leq r< 2\sqrt{\frac{P}{X+Y}}} 
        U_{k-2}\lrp{\frac{r\sqrt{X}}{2\sqrt{P}}} XYB_0\nu(r)\sqrt{4y-r^2} \\ 
        &\quad + \frac{1}{XY} \sum_{1\leq r< 2\sqrt{\frac{P}{X+Y}}}  O(k) 
        \cdot O_\varepsilon\lrp{P^\frac{1}{2} Y^\frac{3}{2} \EE }
        \qquad \text{(by Lemma \ref{lem:bound-sum-class-num})}
        \\
        %%%%%%%%%%%%%%%%%%%%%%%%%%%%%%%%%%%%%%%%
        &= B_0 \sum_{1\leq r< 2\sqrt{\frac{P}{X+Y}}} 
        U_{k-2}\lrp{\frac{r\sqrt{X}}{2\sqrt{P}}}\nu(r)\sqrt{4y-r^2} \\ 
        &\quad +
         O_\varepsilon\bigg(k \Bigg[
        \frac{1}{XY} \cdot \sqrt{\frac{P}{X}} \cdot P^\frac{1}{2} Y^\frac{3}{2} \EE
        \bigg]
        \Bigg)  \\ 
        %%%%%%%%%%%%%%%%%%%%%%%%%%%%%%%%%%%%%%%
        &= B_0 \sum_{1\leq r\le 2\sqrt{y}} U_{k-2}\lrp{\frac{r}{2\sqrt{y}}}\nu(r) \sqrt{4y-r^2} + O\lrp{k\lrb{\frac{PY^\frac{3}{2}}{X^{\frac{5}{2}}}+\frac{P^\frac{1}{2}Y^\frac{1}{2}}{X}}} \\ 
        &\quad +
         O_\varepsilon\bigg(k \Bigg[
        \frac{P Y^\frac{1}{2}}{X^\frac{3}{2}} \EE
        \bigg] 
        \Bigg) \hspace{45mm} \text{(by Lemma \ref{lem:Uk-split-main-error})} \\ 
        %%%%%%%%%%%%%%%%%%%%%%%%%%%%%%%%%%%%%%%
        &= B_0 \sum_{1\leq r\le 2\sqrt{y}} U_{k-2}\lrp{\frac{r}{2\sqrt{y}}}\nu(r) \sqrt{4y-r^2}   +
        O_\varepsilon\lrp{k\lrb{
            \frac{PY^\frac{3}{2}}{X^{\frac{5}{2}}}
            +
            \frac{P^\frac{1}{2}Y^\frac{1}{2}}{X} 
            +
            \frac{P Y^\frac{1}{2}}{X^\frac{3}{2}} \EE
        }} \\
        &= B_0 \sum_{1\leq r\le 2\sqrt{y}} U_{k-2}\lrp{\frac{r}{2\sqrt{y}}}\nu(r) \sqrt{4y-r^2}   +
        O_\varepsilon\lrp{k\lrb{
            \frac{P^\frac{1}{2}Y^\frac{1}{2}}{X} 
            +
            \frac{P Y^\frac{1}{2}}{X^\frac{3}{2}} 
        } \EE
        } \\
        & \hspace{45mm} \text{\Big(since $\tfrac{PY^\frac{3}{2}}{X^{\frac{5}{2}}} \le \tfrac{P Y^\frac{1}{2}}{X^\frac{3}{2}} \le \tfrac{P Y^\frac{1}{2}}{X^\frac{3}{2}} \EE$ and $\tfrac{P^\frac{1}{2}Y^\frac{1}{2}}{X} \le \tfrac{P^\frac{1}{2}Y^\frac{1}{2}}{X} \EE$ \Big)} \\
        &= B_0 \sum_{1\leq r\le 2\sqrt{y}} U_{k-2}\lrp{\frac{r}{2\sqrt{y}}}\nu(r) \sqrt{4y-r^2}   +
        O_\varepsilon \lrp{k\lrb{
            1
            +
            \frac{P^\frac{1}{2}}{X^\frac{1}{2}}
        } \frac{P^\frac{1}{2}Y^\frac{1}{2}}{X} \EE},
    \end{align}
    which constitutes the claimed main term.

    Second, we have that
    \begin{align}
        &|\eqref{eq:initial-double-sum-first-error}| \\
        &=
        \lrabs{
        \frac{1}{XY}\sum_{1\leq r< 2\sqrt{\frac{P}{X+Y}}} \ \sideset{}{^\square}\sum_{\substack{N\in[X,X+Y]\\ P\nmid N}} E(r,N)H_1(r^2N^2-4PN) 
        }\\
        &\le 
        \frac{1}{XY}\sum_{1\leq r< 2\sqrt{\frac{P}{X+Y}}} \ \sideset{}{^\square}\sum_{\substack{N\in[X,X+Y]\\ P\nmid N}} 
        \frac{12k}{4y-r^2} \frac{PY}{X^2}
        \cdot
        |H_1(r^2N^2-4PN)|
        \quad\ \text{(by Lemma \ref{lem:E(r,N)-bound})}
        \\
        &= \frac{1}{XY}\sum_{1\leq r< 2\sqrt{\frac{P}{X+Y}}} \frac{12k}{4y-r^2} \frac{PY}{X^2} \sideset{}{^\square}\sum_{\substack{N\in[X,X+Y]\\ P\nmid N}} H_1(r^2N^2-4PN) \\
        &\qquad \text{\big(since $r^2N^2-4PN \ne 0$ because $NP$ is squarefree\big)} \\
        &\ll_\varepsilon \frac{1}{XY}\sum_{1\leq r< 2\sqrt{\frac{P}{X+Y}}} \frac{12k}{4y-r^2} \frac{PY}{X^2} \cdot  
            XY \sqrt{4y-r^2} \EE \qquad\qquad \text{(by Lemma \ref{lem:bound-sum-class-num})}
        \\
        &= \frac{12kPY}{X^2} 
            \EE
        \sum_{1\leq r< 2\sqrt{\frac{P}{X+Y}}} \frac{1}{\sqrt{4y-r^2}} \\
        &\le \frac{12kPY}{X^2} 
            \EE
        \lrb{
            \int_1^{2 \sqrt{\frac{P}{X+Y}}} \frac{1}{\sqrt{4y-r^2}} dr
            + 
            \frac{1}{\sqrt{4y-4\frac{P}{X+Y}}}
        } \\
        &\le \frac{12kPY}{X^2} 
            \EE
        \lrb{
            \int_0^{2 \sqrt{y}} \frac{1}{\sqrt{4y-r^2}} dr
            + 
            \frac{1}{2 \sqrt{\frac{P}{X}} \sqrt{1-\lrp{1+\frac{Y}{X}}^{-1}}}
        } \\
        &\le \frac{12kPY}{X^2} 
            \EE
        \lrb{
            \frac{\pi}{2}
            + 
            \frac{1}{2 \sqrt{\frac{P}{X}} \cdot 
            \sqrt{\frac{1}{2} \frac{Y}{X}}
            }
        } 
        \qquad 
        \text{(by a first-order Taylor approximation)}
        \\
        &\ll_\varepsilon  k 
        \lrb{
            \frac{PY}{X^2}
            + 
            \frac{P^\frac{1}{2}Y^\frac{1}{2}}{X}
        } \EE \\
        &= k \lrb{\frac{P^\frac{1}{2}Y^\frac{1}{2}}{X}+1}\frac{P^\frac{1}{2}Y^\frac{1}{2}}{X} \EE \\
        &\ll_\varepsilon  k \lrb{\frac{P^\frac{1}{2}}{X^\frac{1}{2}}+1}\frac{P^\frac{1}{2}Y^\frac{1}{2}}{X} \EE,
    \end{align}
    which is bounded by the claimed error term.

    Third, observe that for $r$ in the range $2\sqrt{\frac{P}{X+Y}}\leq r\leq 2\sqrt{\frac{P}{X}}$ and $N\in[X,X+Y],$ we have the bound
    \begin{align}
        &\lrabs{r^2N^2-4PN} \leq 4PN\lrp{1-\frac{N}{X+Y}}\leq 4P(X+Y)\lrp{1-\frac{X}{X+Y}}\ll PY, \\
        &\text{so that} \qquad
        H_1(r^2N^2-4PN) \ll_\varepsilon  (PY)^{\frac{1}{2}+\varepsilon}. \label{eqn:H-bound-from-r2N24pN-disc-bound}
    \end{align}
    Hence
    we can bound \eqref{eq:initial-double-sum-second-error} by 
    \begin{align}
        \eqref{eq:initial-double-sum-second-error}
        &= \frac{1}{XY}
        \sideset{}{^\square}\sum_{\substack{N\in[X,X+Y] \\ P\nmid N}}
        \sum_{2\sqrt{\frac{P}{X+Y}}\le r \le 2\sqrt{\frac{P}{N}}}  U_{k-2}\lrp{\frac{r\sqrt{N}}{2\sqrt{P}}} H_1(r^2N^2-4PN) \\
        &\ll_\varepsilon  \frac{1}{XY} \sideset{}{^\square}\sum_{\substack{N\in [X,X+Y]\\P\nmid N}} \sum_{2\sqrt{\frac{P}{X+Y}}\le r \le 2\sqrt{\frac{P}{N}}} k(PY)^{\frac{1}{2}+\varepsilon}  
        \qquad  \text{\big(since $U_{k-2}(x)\ll k$ and by \eqref{eqn:H-bound-from-r2N24pN-disc-bound}\big)}
        \\ 
        & \ll_\varepsilon  \frac{1}{XY}\cdot Y\lrb{2\sqrt{\frac{P}{X}}-2\sqrt{\frac{P}{X+Y}}+1} k(PY)^{\frac{1}{2}+\varepsilon} \\ 
        &\ll_\varepsilon  \frac{k}{X} \lrb{\sqrt{\frac{P}{X}}
        \lrp{1-\sqrt{\frac{X}{X+Y}}}+1} (PY)^{\frac{1}{2}+\varepsilon} \\
        & \le \frac{k}{X} \lrb{\sqrt{\frac{P}{X}} \lrp{\frac{1}{2} \frac{Y}{X}} + 1} (PY)^{\frac{1}{2}+\varepsilon} 
        \qquad \text{(by a first order Taylor approximation)}
        \\
        &\ll_\varepsilon   k\lrb{\frac{PY^{\frac{3}{2}}}{X^{\frac{5}{2}}}+\frac{P^\frac{1}{2}Y^\frac{1}{2}}{X}}(PY)^\varepsilon \\
        &\ll_\varepsilon  k\lrb{\frac{PY^\frac{1}{2}}{X^{\frac{3}{2}}}+\frac{P^\frac{1}{2}Y^\frac{1}{2}}{X}}(PX)^\varepsilon \\
        &\ll_\varepsilon  k\lrb{
            \frac{P^\frac{1}{2}}{X^\frac{1}{2}} 
            + 1
        }\frac{P^\frac{1}{2} Y^\frac{1}{2}}{X} \EE, 
    \end{align}
    which is again bounded by the claimed error term. This completes the proof.
\end{proof}

In the above proof, we used the following lemma.
\begin{lemma} \label{lem:bound-sum-class-num}
    Let $P > 2$ be prime and let $[X, X+Y]$ be an interval of length $Y = o(X)$. Assume further that $r < \sqrt{\frac{4P}{X+Y}}$. Then
    \begin{align}
        \sideset{}{^\square} \sum_{\substack{N \in [X, X+Y] \\ P \nmid N}} H_1(r^2N^2-4PN) &= XYB_0\nu(r) \sqrt{4y-r^2}  + O_\varepsilon \lrp{P^\frac{1}{2} Y^\frac{3}{2} \EE} \\[-15pt]
        &\ll_\varepsilon  
        XY \sqrt{4y-r^2} \,\EE,
    \end{align}
    where
    \begin{align}  
        \EE = \EE(P,X,Y,\varepsilon) :=
        \lrb{
            1 
            + 
            \frac{
                P^\frac{1}{10} X^\frac{3}{5} 
            }{
                Y^\frac{9}{10}
            }
            +
            \frac{X}{Y^\frac{11}{9}}
            +
            \frac{X^\frac{1}{2}}{Y^\frac{11}{18}}
        }
        (PX)^\varepsilon,
    \end{align}
    as defined in Proposition \ref{prop:main-proposition}.
\end{lemma}
This lemma is a restatement of \cite[Proposition 3.6]{Zubrilina}, using the three facts that
    \begin{align}
        \sqrt{4PX-r^2X^2}&=X\sqrt{4y-r^2}, \\
        \frac{P^\frac{1}{2} Y^2}{X^\frac{1}{2}}+rX^\frac{1}{2}Y^\frac{3}{2} 
        &\le 
        \frac{P^\frac{1}{2}X^\frac{1}{2}Y^\frac{3}{2}}{X^\frac{1}{2}} + 
        2\frac{P^\frac{1}{2}}{X^\frac{1}{2}} 
        \cdot
        X^\frac{1}{2}Y^\frac{3}{2}
        \le 3 P^\frac{1}{2}Y^\frac{3}{2}, \\
        \text{and } \ 
        XY\sqrt{4y-r^2} &\ge XY\sqrt{4 \frac PX - 4 \frac{P}{X+Y}} \\
        &= XY\sqrt{4 \frac{P}{X}\lrp{1- \frac{X}{X+Y}} } \ge XY\sqrt{4 \frac{P}{X} \cdot \frac14 \frac{Y}{X} } = P^\frac{1}{2} Y^\frac{3}{2}.
    \end{align}
We cite \cite[Proposition 3.6]{Zubrilina} here for reference.

\begin{unnumberedlemma}[{\cite[Proposition 3.6]{Zubrilina}}]
    \label{lem:bound-avg-class-number}
    Let $P > 2$ be prime and let $[X, X+Y]$ be an interval of length $Y = o(X)$. Assume further that $r^2(X+Y)<4P$. Then
    \begin{align}
        \sideset{}{^\square} \sum_{\substack{N \in [X, X+Y] \\ P \nmid N}} H_1(r^2N^2-4PN) 
        &= YB_0\nu(r) \sqrt{4PX-r^2X^2} \\[-15pt]
        &\hspace{-20mm} + O_\varepsilon \lrp{\bigg[ P^\frac{3}{5}X^\frac{3}{5}Y^\frac{3}{5} + \frac{P^\frac{1}{2}Y^2}{X^\frac{1}{2}} + rX^\frac{1}{2}Y^\frac{3}{2} + P^\frac{1}{2}XY^\frac{5}{18} + P^\frac{1}{2}Y^\frac{8}{9}X^\frac{1}{2} \bigg] (PX)^\varepsilon}.
    \end{align}
\end{unnumberedlemma}

\begin{remark}
    Some modifications are needed in the proof of
    \cite[Proposition 3.6]{Zubrilina}. Specifically, the Euler-product
    evaluation in the second displayed equation of
    \cite[Lemma 3.10]{Zubrilina} includes a local factor at $P$, whereas
    \cite[Lemma 3.8]{Zubrilina}, which is used to compute these local
    factors, only applies to parameters coprime to $P$.
    The result needed later in the paper can be recovered by the
    following modifications to \cite[Lemma 3.10]{Zubrilina}.

    Let $B^{(P)}$ denote the value $B$ (from the second displayed equation of \cite[Lemma 3.10]{Zubrilina}) with the Euler factor at $P$ removed: $B^{(P)} := B / \lrp{\frac{P^4-2P^2-P+1}{(P^2-1)^2}}$. Additionally let $\mathcal{T}_r^{(P)}$ denote the elements of $\mathcal{T}_r$ away from $P$: $\mathcal{T}_r^{(P)} := \{(m,d,g) \in \mathcal{T}_r : P \nmid mdg\}$.

    First, one can modify the statement of \cite[Lemma 3.10]{Zubrilina}  to claim the identity
    \begin{align}
        \label{eqn:temp-theta-sum-over-TP-equals-BP}
        \sum_{\substack{(m,d,g) \in \mathcal{T}_r^{(P)}}} \Theta_r(m,d,g) = B^{(P)} \nu(r),
    \end{align}
    instead of the unrestricted claim that
    \begin{align}
        \label{eqn:temp-theta-sum-over-T-equals-B}
        \sum_{\substack{(m,d,g) \in \mathcal{T}_r}} \Theta_r(m,d,g) = B \nu(r).
    \end{align}
    One should also add the conditions that $P\nmid r$ and $Z,Z'<P$. (Note that for our later choices $Z=Y^\tau$ and $Z'=Y^\sigma$ in the proof of \cite[Proposition 3.12]{Zubrilina},
    these two conditions will be satisfied 
    since $r^2(X+Y)<4P$ implies that $r < P$ and $Z,Z' \le Y \le \frac{1}{4} X < P$ for large enough $X$.)

    Second, one can modify the proof of \cite[Lemma 3.10]{Zubrilina} to show \eqref{eqn:temp-theta-sum-over-TP-equals-BP} instead of \eqref{eqn:temp-theta-sum-over-T-equals-B}. By the same tail estimates as in \cite[Lemma 3.10]{Zubrilina}, this will then imply that
        \begin{align}
            &\sum_{\substack{(m,d,g) \in \mathcal{T}_r^{(P)} \\
            mg \le Z'\!,\, d \le Z}} \Theta_r(m,d,g)  - B^{(P)} \nu(r) \ll Z^{-2} + (Z')^{\frac{-1}{5}} \label{eqn:temp-modified-tail-estimate-Zub-Lemma-3.10} \\
            &\text{(c.f. the third displayed equation of \cite[Lemma 3.10]{Zubrilina})}.
        \end{align}
        
        Third, since $Z,Z' < P$, note that we have the two facts
        \begin{align}
            \sum_{\substack{(m,d,g) \in \mathcal{T}_r \\
            mg \le Z'\!,\, d \le Z}} \Theta_r(m,d,g) 
            &= \sum_{\substack{(m,d,g) \in \mathcal{T}_r^{(P)} \\
            mg \le Z'\!,\, d \le Z}} \Theta_r(m,d,g), \\
            \text{and} \qquad 
            \lrabs{B \nu(r) - B^{(P)}\nu(r)} &\ll P^{-3} \ll Z^{-3} \ll Z^{-2}.
        \end{align}
        Applying these two facts to \eqref{eqn:temp-modified-tail-estimate-Zub-Lemma-3.10} yields
        \begin{align}
            &\sum_{\substack{(m,d,g) \in \mathcal{T}_r \\
            mg \le Z'\!,\, d \le Z}} \Theta_r(m,d,g)  - B \nu(r) \ll Z^{-2} + (Z')^{\frac{-1}{5}},
        \end{align}
        which verifies the desired result.
\end{remark}

In the proof of Proposition \ref{prop:main-proposition}, we also used the following two lemmas.
\begin{lemma}\label{lem:Uk-split-main-error}
    Let $P > 2$ be prime and let $[X, X+Y]$ be an interval of length $Y = o(X)$. Then
    \begin{align}
        & \sum_{1\leq r < 2\sqrt{\frac{P}{X+Y}}} 
        U_{k-2}\lrp{\frac{r\sqrt{X}}{2\sqrt{P}}}\nu(r)\sqrt{4y-r^2} \\
        &= \sum_{1\leq r\leq 2\sqrt{y}} U_{k-2}\lrp{\frac{r}{2\sqrt{y}}}\nu(r) \sqrt{4y-r^2} + O\lrp{k\lrb{\frac{PY^{\frac{3}{2}}}{X^{\frac{5}{2}}}+\frac{P^\frac{1}{2}Y^\frac{1}{2}}{X}}} \\
    \end{align}
\end{lemma}
\begin{proof}
    First, write
    \begin{align}
        &\sum_{1\leq r<2\sqrt{\frac{P}{X+Y}}} 
        U_{k-2}\lrp{\frac{r\sqrt{X}}{2\sqrt{P}}}\nu(r)\sqrt{4y-r^2} \\
        %%%%%%%%%%%%%%%%%%%%%%%%%%%%%%%%%%%%%%%%%%%
        &=  \sum_{1\leq r\leq 2\sqrt{\frac{P}{X}}} 
        U_{k-2}\lrp{\frac{r\sqrt{X}}{2\sqrt{P}}}\nu(r)\sqrt{4y-r^2} 
        % \label{eq:small-lemma-main} 
        \ \ - \hspace{-2mm} \sum_{2\sqrt{\frac{P}{X+Y}}\leq r\leq 2\sqrt{\frac{P}{X}}} 
        U_{k-2}\lrp{\frac{r\sqrt{X}}{2\sqrt{P}}}\nu(r)\sqrt{4y-r^2}. 
        \\
        &= \sum_{1\leq r\leq 2\sqrt{y}} 
        U_{k-2}\lrp{\frac{r}{2\sqrt{y}}}\nu(r)\sqrt{4y-r^2} 
        \ \ - \hspace{-2mm}
        \sum_{2\sqrt{\frac{P}{X+Y}}\leq r\leq 2\sqrt{\frac{P}{X}}} 
        U_{k-2}\lrp{\frac{r\sqrt{X}}{2\sqrt{P}}}\nu(r)\sqrt{4y-r^2}.
    \end{align}
    Observe that the first summation is the desired main term of the lemma. Moreover, the second summation is bounded by
    \begin{align} 
        &\sum_{2\sqrt{\frac{P}{X+Y}}\leq r \le 2\sqrt{\frac{P}{X}}} U_{k-2}\lrp{\frac{r\sqrt{X}}{2\sqrt{P}}}\nu(r)\sqrt{4y-r^2} \\
        &\ll \lrp{2\sqrt{\frac{P}{X}}-2\sqrt{\frac{P}{X+Y}}+1}
        \cdot k \cdot
        \frac{P^\frac{1}{2}Y^\frac{1}{2}}{X} 
        \label{eqn:temp_+1_for_range_length}
        \\
        &\quad\    \text{\Big(since $U_{k-2}(x)\ll k,  \nu(r)\ll 1$, and $4y-r^2 \le 4 \tfrac{P}{X} - 4 \tfrac{P}{X+Y} = \tfrac{4P}{X} \lrp{1- \tfrac{X}{X+Y}} \le \tfrac{4PY}{X^2}$ \Big)} \\
        &\ll k\lrp{\frac{PY^{\frac{3}{2}}}{X^{\frac{5}{2}}}+\frac{P^\frac{1}{2}Y^\frac{1}{2}}{X}},
        \\
        &\qquad \text{ \Big(since $\sqrt{\tfrac{P}{X}}-\sqrt{\tfrac{P}{X+Y}} = \sqrt{\tfrac{P}{X}} \Big(1 - \sqrt{\tfrac{X}{X+Y}}\Big)\leq \sqrt{\tfrac{P}{X}} \cdot \tfrac{1}{2} \tfrac{Y}{X} = \tfrac 12 \tfrac{ P^\frac{1}{2} Y}{X^\frac{3}{2}}$\Big)}
        % k\frac{PY}{X^2}+k\frac{P^{\frac{1}{2}}}{X^{\frac{1}{2}}}, \qquad  \\
    \end{align}
    yielding the desired result.
\end{proof}

\begin{remark}
    If one compares the claims of \cite[Section 3.3.7]{Zubrilina} (and also \cite[Proposition 3.2]{Zubrilina}) with Lemma \ref{lem:Uk-split-main-error}, one might notice that we have included an extra error term. Specifically, the second displayed equation of \cite[Section 3.3.7]{Zubrilina} does not include the error term corresponding to $O\lrp{k\frac{P^\frac{1}{2} Y^\frac{1}{2}}{X}}$ from Lemma \ref{lem:Uk-split-main-error}. It seems that this omission comes from not including a ``$+1$" corresponding to the one in \eqref{eqn:temp_+1_for_range_length}. 
    In fact, this ``$+1$" needs to be retained because it will actually be the dominant term of \eqref{eqn:temp_+1_for_range_length} for the values of $P,X,Y$ we will choose shortly. The same remark also applies to the third displayed equation of \cite[Section 3.3.7]{Zubrilina}.
\end{remark}

\begin{lemma} \label{lem:E(r,N)-bound}
    Let $E(r,N)$ be defined as in \eqref{eqn:def-E(r,N)} and let $y = \frac{P}{X}$. Then
    \begin{align}
        |E(r,N)| \le \frac{12k}{4y-r^2} \frac{PY}{X^2}
    \end{align}
\end{lemma}
\begin{proof}
    By Lemma \ref{lem:lipschitz-const-at-x0-for-U}, we have that
    \begin{align}
        |E(r,N)| 
        &=
        \lrabs{
        U_{k-2}\lrp{\frac{r \sqrt{N}}{2\sqrt{P}}}-
        U_{k-2}\lrp{\frac{r\sqrt{X}}{2\sqrt{P}}}
        }
        \\
        &\le \frac{3(k-1)}{1-\lrabs{\frac{r \sqrt X}{2 \sqrt P}}} \cdot \lrb{\frac{r\sqrt N}{2 \sqrt P} - \frac{r\sqrt X}{2 \sqrt P}} \\
        &\le \frac{3k\lrp{1+\lrabs{\frac{r}{2\sqrt y}}}}{1-\lrp{\frac{r}{2 \sqrt y}}^2} \cdot \lrb{\frac{r\sqrt{X+Y}}{2 \sqrt P} - \frac{r\sqrt X}{2 \sqrt P}} \\
        &\le \frac{3k \cdot 2 \cdot 4y}{4y-r^2} \cdot \frac{r}{2 \sqrt P} \lrb{\sqrt{X+Y} - \sqrt{X}} \\
        &\le \frac{24ky}{4y-r^2} \cdot \frac{r}{2 \sqrt P} \lrb{\sqrt{X} \cdot \frac{1}{2} \frac{Y}{X}} \\
        &= \frac{24k \frac PX}{4y-r^2} \cdot \frac{r}{2 \sqrt y} \lrb{\frac{1}{2} \frac{Y}{X}} \\
        &\le \frac{12k}{4y-r^2} \frac{PY}{X^2},
    \end{align}
    as desired.
\end{proof}

In the proof of Lemma \ref{lem:E(r,N)-bound}, we used the following fact.
\begin{lemma} \label{lem:lipschitz-const-at-x0-for-U}
    Fix a natural number $K \ge 1$, and let $x_0 \in (-1,1)$. Then for all $x \in [-1,1]$,
    \begin{align}
        \label{eqn:lipschitz-const-goal}
        |U_{K-1}(x) - U_{K-1}(x_0)| \le \frac{3K}{1-|x_0|} |x-x_0|.
    \end{align}
\end{lemma}
\begin{proof}
We divide into two cases.

\textbf{Case 1: $|x-x_0| > \frac{2}{3} (1-|x_0|)$. \\}
Recall the global bound for Chebyshev polynomials
\begin{align}
    -K \le U_{K-1}(x) \le K.
\end{align}
This bound then implies that
\begin{align}
    |U_{K-1}(x)-U_{K-1}(x_0)| \le 2K = \frac{3K}{1-|x_0|} \cdot \frac{2}{3} (1-|x_0|) \le \frac{3K}{1-|x_0|} |x-x_0|,
\end{align}
verifying \eqref{eqn:lipschitz-const-goal}.

\textbf{Case 2: $|x-x_0| \le \frac{2}{3} (1-|x_0|)$. \\}
Observe that for every $t$ between $x_0$ and $x$, 
\begin{align}
    1-|t|
    &\ge
    1-|x_0|-|t-x_0| \\
    &\ge
    1-|x_0|-\frac{2}{3}(1-|x_0|) \\
    &=
    \frac{1}{3}(1-|x_0|).
    \label{eq:xi-away-from-endpoint}
\end{align}

Additionally, we have the derivative bound
\begin{align}
    |U_{K-1}'(t)|
    &=
    \lrabs{
        \frac{1}{-\sin \theta}
        \cdot
        \frac{d}{d\theta}
        \lrb{
            \frac{\sin(K\theta)}{\sin \theta}
        }
    } \qquad \qquad\quad
    \text{(where $t=\cos \theta$)}
    \\
    &=
    \lrabs{
    \frac{
        K\cos(K\theta)\sin\theta
        -
        \sin(K\theta)\cos\theta
    }{\sin^3\theta}
    } \\
    &\le
    \frac{
        K\sin\theta
        +
        K\sin\theta\,|\cos\theta|
    }{\sin^3\theta}
    \qquad\qquad\quad
    \text{(since $|\sin(K\theta)|\le K\sin\theta$)}
    \\
    &=
    \frac{K(1+|t|)}{1-t^2} \\
    &=
    \frac{K}{1-|t|}.
    \label{eq:U-derivative-bound}
\end{align}

Hence by the Mean Value Theorem, we have that
\begin{align}
    |U_{K-1}(x)-U_{K-1}(x_0)|
    &=
    |U_{K-1}'(t)|\,|x-x_0| 
    \qquad \text{(for some $t$ between $x_0$ and $x$)} 
    \\
    &\le
    \frac{K}{1-|t|}|x-x_0| 
    \qquad\quad\ \text{(by \eqref{eq:U-derivative-bound})}
    \\
    &\le
    \frac{3K}{1-|x_0|}|x-x_0|, 
    \qquad\ \,  \text{(by \eqref{eq:xi-away-from-endpoint})}
\end{align}
verifying \eqref{eqn:lipschitz-const-goal}.
\end{proof}

%%%%%%%%%%%%%%%%%%%%%%%%%%%%%%%%%%%%%%%%%%%%%%%%%%%%%%
%%%%%% PROVING THEOREM ONE FROM MODIFIED STUFF %%%%%%%
%%%%%%%%%%%%%%%%%%%%%%%%%%%%%%%%%%%%%%%%%%%%%%%%%%%%%%

\section{Proof of Theorem \ref{thm:main-result}}
\label{sec:proof-of-main-thm}

We start by proving Proposition \ref{prop:pre-main-result} (c.f. \cite[Equation (20)]{Zubrilina}), which will then immediately imply Theorem \ref{thm:main-result}.
For the proof of Proposition \ref{prop:pre-main-result}, we will utilize Proposition \ref{prop:main-proposition}, the Skoruppa-Zagier trace formula (Lemma \ref{lem:skoruppa-zagier-trace-formula}), and a class number estimate (Lemma \ref{lem:zubrilina-prop-3.1}) shown in \cite{Zubrilina}.

\begin{lemma}[{\cite[Section 2, equations (5) and (7)]{Skoruppa-Zagier}}] \label{lem:skoruppa-zagier-trace-formula}
For squarefree integers $N > 1$ and primes $P\nmid N$, we have that 
    \begin{align}
        &\sum_{f \in H^\new(N,k)}\sqrt{P}\lambda_f(P)\varepsilon(f) \\ 
        &= \frac{H_1(-4PN)}{2} + (-1)^{\frac{k}{2}-1}\sum_{1 \le r \le 2\sqrt{\frac{P}{N}}}U_{k-2}\lrp{\frac{r\sqrt{N}}{2\sqrt{P}}}H_1(r^2N^2-4PN) - \delta_{k=2}(P+1).
    \end{align}
\end{lemma}

For the following lemma, recall that $A_0$ is the value that was defined in \eqref{eqn:def-A0}.
\begin{lemma}[{\cite[Proposition 3.1]{Zubrilina}}] \label{lem:zubrilina-prop-3.1}
    Let $P > 2$ be prime and let $[X, X+Y]$ be an interval of length $Y = o(X)$. Let $y := \frac{P}{X}$. Then as $X \to \infty$, 
    \begin{align}
        &\frac{1}{XY} \sideset{}{^\square}\sum_{\substack{N \in [X, X+Y] \\ P \nmid N}} \lrp{\frac{h(-PN)}{2} + \frac{h(-4PN)}{2}} \\
        &= A_0\sqrt{y} + O_\varepsilon\lrp{\bigg[\frac{1}{P^\frac{1}{2}X^\frac{1}{2}}+\frac{P^\frac{11}{19}}{Y^\frac{16}{19}}+\frac{P^\frac{1}{2}Y}{X^\frac{3}{2}}\bigg](PX)^\varepsilon}.
    \end{align}
\end{lemma}
\begin{remark}
    It seems that two minor modifications are needed for the proof of \cite[Proposition 3.1]{Zubrilina}. First, it seems that the definition $T:=P^\frac{5}{12}X^\frac{-1}{12} Y^\frac{5}{6}$ just after \cite[Equation (7)]{Zubrilina} should be changed to  $T := P^\frac{8}{19} Y^\frac{16}{19}$. 
    Second, it seems that the bounding term $X^{\frac{3}{5}+\varepsilon}$ in \cite[Equation (5)]{Zubrilina} should be improved to $X^{\frac{1}{2}}$ (which is possible by \cite[Theorem 3.5]{Zubrilina}).
\end{remark}

We now prove Proposition \ref{prop:pre-main-result}.

\begin{proposition} \label{prop:pre-main-result}
Fix an even integer $k\ge 2$. Let $P$, $X$, and $Y$ be parameters going to infinity with $P$ prime; assume further that $Y=(1+o(1))X^{1-\delta_2}$ and $P\ll X^{1+\delta_1}$ for some $\delta_1,\delta_2>0$ with $0<\delta_1<\frac{1}{11}$ and $2\delta_1<\delta_2<\frac{1}{13}(4-18\delta_1)$. Let $y:= \frac{P}{X}$. Then 
\begin{align}
    &\frac{1}{XY} \sideset{}{^\square} \sum_{N \in [X, X+Y]} \sum_{f \in H^\new(N,k)} \lambda_f(P)\sqrt{P}\varepsilon(f) \\
    &= (-1)^{\frac{k}{2}-1} B_0\!\!\!\sum_{1 \le r \le 2\sqrt{y}}   U_{k-2}\lrp{\frac{r}{2\sqrt{y}}}\nu(r)\sqrt{4y-r^2} 
    + A_0\sqrt{y} - \frac{\delta_{k=2}}{\zeta(2)} y  + O_\varepsilon\lrp{kX^{-\delta'+\varepsilon}+\frac{k}{P}}
\end{align}
where 
$
    \delta' := \frac{\delta_2}{2} - \delta_1 +\min\lrcb{0,\frac{2}{9}-\frac{11\delta_2}{9}} > 0
$.
\end{proposition}
\begin{proof}
By the Skoruppa-Zagier trace formula (Lemma \ref{lem:skoruppa-zagier-trace-formula}), we have that
\begin{align}
    &\frac{1}{XY} \sideset{}{^\square} \sum_{\substack{N \in [X, X+Y]}} \sum_{f \in H^\new(N,k)} \sqrt{P}\lambda_f(P)\varepsilon(f) \\
    &=
    \frac{1}{XY} \sideset{}{^\square} \sum_{\substack{N \in [X, X+Y] \\ P \mid N}} \sum_{f \in H^\new(N,k)} \sqrt{P}\lambda_f(P)\varepsilon(f) 
    \label{eqn:trace-goal-term1}
    \\
    &\quad +\frac{1}{XY} \sideset{}{^\square}\sum_{\substack{N \in [X, X+Y] \\ P \nmid N}}  \frac{H_1(-4PN)}{2} 
    \label{eqn:trace-goal-term2}
    \\
    &\quad + \frac{1}{XY} \sideset{}{^\square}\sum_{\substack{N \in [X, X+Y] \\ P \nmid N}}
    (-1)^{\frac{k}{2}-1}\!\!\! 
    \sum_{1 \le r \le 2\sqrt{\frac{P}{N}}}U_{k-2}\lrp{\frac{r\sqrt{N}}{2\sqrt{P}}}H_1\lrp{r^2N^2-4PN} 
    \label{eqn:trace-goal-term3}
    \\
    &\quad + \frac{1}{XY} \sideset{}{^\square}\sum_{\substack{N \in [X, X+Y] \\ P \nmid N}} -\delta_{k=2}(P+1). 
    \label{eqn:trace-goal-term4}
\end{align}
We deal with the four summations \eqref{eqn:trace-goal-term1}, \eqref{eqn:trace-goal-term2}, \eqref{eqn:trace-goal-term3}, and \eqref{eqn:trace-goal-term4} separately. 

First, we have that
\begin{align}
    |\eqref{eqn:trace-goal-term1}| 
    &\le 
    \frac{1}{XY} \sideset{}{^\square} \sum_{\substack{N \in [X, X+Y] \\ P \mid N}} \sum_{f \in H^\new(N,k)} \lrabs{\sqrt{P}\lambda_f(P)\varepsilon(f)} \\
    &= \frac{1}{XY} \sideset{}{^\square} \sum_{\substack{N \in [X, X+Y] \\ P \mid N}} \sum_{f \in H^\new(N,k)} 1 \qquad \text{(since $P \parallel N$, by \cite[Proposition 13.3.14]{Cohen-Stromberg})} \\
    &\ll \frac{1}{XY} \sideset{}{^\square} \sum_{\substack{N \in [X, X+Y] \\ P \mid N}} kN 
    \qquad\qquad\qquad \text{(by \eqref{eqn:dim-formula-for-SknewN})}
    \\
    &\ll \frac{1}{XY} \lrb{\frac{Y}{P} + 1} \cdot k X \\
    &\ll \frac{k}{Y} + \frac{k}{P}.
\end{align}

Second, using the fact that
\begin{align}
    H_1(-d)=\sum_{ f^2\mid d}h\lrp{\frac{-d}{f^2}}+O(1), \qquad \text{(by 
    % the definition of $H(\cdot)$ 
    \cite[Definition 3.4.18]{Cohen-Stromberg})}
\end{align}
we have 
\begin{align}
    \eqref{eqn:trace-goal-term2}
    &= \frac{1}{XY} \sideset{}{^\square} \sum_{\substack{N \in [X, X+Y] \\ P \nmid N}}
    \lrp{
        \sum_{f^2 \mid 4PN} h\lrp{\frac{-4PN}{f^2}} + O(1) 
    }
    \\
    &= \frac{1}{XY} \sideset{}{^\square} \sum_{\substack{N \in [X, X+Y] \\ P \nmid N}} \lrp{\frac{h(-PN)}{2}+\frac{h(-4PN)}{2}+O(1)}
    \quad \text{(since $PN$ is squarefree)}
    \\
    &= A_0\sqrt{y} + O_\varepsilon\lrp{\bigg[\frac{1}{P^\frac{1}{2}X^\frac{1}{2}} + \frac{P^{\frac{11}{19}}}{Y^{\frac{16}{19}}} + \frac{P^\frac{1}{2}Y}{X^{\frac{3}{2}}}\bigg] X^\varepsilon} + O\lrp{\frac{1}{X}} \qquad \text{(by Lemma \ref{lem:zubrilina-prop-3.1})} \\
    &= A_0\sqrt{y} + O_\varepsilon\lrp{\bigg[\frac{1}{X^\frac{1}{2}} + \frac{P^{\frac{11}{19}}}{Y^{\frac{16}{19}}} + \frac{P^\frac{1}{2}Y}{X^{\frac{3}{2}}}\bigg] X^\varepsilon}.
\end{align}

Third, by Proposition \ref{prop:main-proposition}, we have that
\begin{align}
    \eqref{eqn:trace-goal-term3} 
    & =(-1)^{\frac{k}{2}-1} B_0\!\!\!\sum_{1\leq r\leq 2\sqrt{y}}U_{k-2}\lrp{\frac{r}{2\sqrt{y}}}\nu(r)\sqrt{4y-r^2} 
       +
       O_\varepsilon\lrp{ k
        \lrb{
            \frac{P^\frac{1}{2}}{X^\frac{1}{2}}
            + 1
        }
        \frac{P^\frac{1}{2} Y^\frac{1}{2}}{X}
        \EE
       }.
\end{align}

Fourth, we have that
\begin{align}
    \eqref{eqn:trace-goal-term4} 
    &= -\delta_{k=2}  (P+1) \frac{1}{XY} \sideset{}{^\square}\sum_{\substack{N \in [X, X+Y] \\ P \nmid N}} 1 \\
    &= -\delta_{k=2}  (P+1) \frac{1}{XY} \sum_{\substack{N \in [X, X+Y]}} \mu^2(N) + O\lrp{\frac{P}{XY}\lrb{\frac{Y}{P}+1}} \\
    &= -\delta_{k=2}  (P+1) \frac{1}{XY} 
    \lrb{
        \frac{Y}{\zeta(2)} + O\lrp{\sqrt X}
    }
    + O\lrp{\frac{1}{X} + \frac{P}{XY}} \\
    &\qquad \text{(since $\textstyle \sum_{N\le Z} \mu^2(N) = \frac{Z}{\zeta(2)} + O(\sqrt Z)$ by \cite[Exercise 4.4.10]{murty})} \\
    &= -\frac{\delta_{k=2}}{\zeta(2)}   \frac{P+1}{X}
    + O\lrp{\frac{1}{X} + \frac{P}{XY} + \frac{P}{X^\frac{1}{2}Y}} \\
    &= -\frac{\delta_{k=2}}{\zeta(2)}   y
    + O\lrp{\frac{1}{X} + \frac{P}{X^\frac{1}{2}Y}}.
\end{align}

Combining these four estimates for \eqref{eqn:trace-goal-term1}, \eqref{eqn:trace-goal-term2}, \eqref{eqn:trace-goal-term3}, and \eqref{eqn:trace-goal-term4}, we obtain that
\begin{align}
    & \frac{1}{XY} \sideset{}{^\square}\sum_{N \in [X, X+Y]}\sum_{f \in H^\new(N,k)} \sqrt{P}\lambda_f(P)\varepsilon(f) \\
    & = (-1)^{\frac{k}{2}-1} B_0\!\!\!\sum_{1 \le r \le 2\sqrt{y}}U_{k-2}\lrp{\frac{r}{2\sqrt{y}}} \nu(r)\sqrt{4y-r^2} + A_0\sqrt{y} - \frac{\delta_{k=2}}{\zeta(2)} y \\
    &\quad + O_\varepsilon\lrp{\bigg[\frac{1}{X^\frac{1}{2}} + \frac{P^{\frac{11}{19}}}{Y^{\frac{16}{19}}} + \frac{P^\frac{1}{2}Y}{X^{\frac{3}{2}}}\bigg]X^\varepsilon}  
    + O_\varepsilon\lrp{ k
        \lrb{
            \frac{P^\frac{1}{2}}{X^\frac{1}{2}}
            + 1
        }
        \frac{P^\frac{1}{2} Y^\frac{1}{2}}{X}
        \EE
       } \\
    &\quad + O\lrp{\frac{1}{X} + \frac{P}{X^{\frac{1}{2}}Y}} +  O\lrp{\frac{k}{Y}}+O\lrp{\frac{k}{P}}.
    \label{eqn:final-big-O-error-terms-PXY}
\end{align}

All that remains is to verify that the first four big-$O$ error terms of \eqref{eqn:final-big-O-error-terms-PXY} are each bounded by $k X^{-\delta'+\varepsilon}$. Recall that $P=O(X^{1+\delta_1})$, and that $Y=(1+o(1))X^{1-\delta_2}$. This then means that the first big-$O$ error term is bounded by
\begin{align}
    \bigg[\frac{1}{X^\frac{1}{2}} + \frac{P^{\frac{11}{19}}}{Y^{\frac{16}{19}}} + \frac{P^\frac{1}{2}Y}{X^{\frac{3}{2}}}\bigg]X^\varepsilon
    &\ll_\varepsilon 
    \lrb{
        X^\frac{-1}{2} + X^{\frac{11}{19} (1+\delta_1) - \frac{16}{19}(1-\delta_2) } + 
        X^{(1-\delta_2) + \frac{1}{2}(1+\delta_1) - \frac{3}{2}}
    }X^\varepsilon \\
    &=
    \lrb{
        X^\frac{-1}{2} + X^{
            \frac{-5}{19} + \frac{11 \delta_1}{19}  + \frac{16\delta_2}{19} 
        } + 
        X^{-\delta_2 + \frac{\delta_1}{2}}
    }X^\varepsilon \\
    &\ll_\varepsilon 
    \lrb{
        X^{
            \frac{-5}{19} + \frac{11 \delta_1}{19}  + \frac{16\delta_2}{19} 
            +
            \,\lrp{
                \frac{7}{171} + \frac{8 \delta_1}{19} - \frac{41\delta_2}{342}
            }
        } + 
        X^{-\frac{\delta_2}{2} + \delta_1}
    }X^\varepsilon \\
    &=
    \lrb{
        X^{-\frac{\delta_2}{2} +\delta_1 + \,\lrp{\frac{-2}{9} + \frac{11\delta_2}{9}}} + 
        X^{-\frac{\delta_2}{2} +\delta_1 }
    }X^\varepsilon \\
    &\ll_\varepsilon 
    \lrb{
        X^{-\frac{\delta_2}{2} +\delta_1 + \max\lrcb{0,\ \frac{-2}{9} + \frac{11\delta_2}{9}}} 
    }X^\varepsilon \\
    &= X^{-\delta'+\varepsilon} \\
    &\ll_\varepsilon  k X^{-\delta'+\varepsilon},
\end{align}
the second big-$O$ error term is bounded by
\begin{align}
    &k
    \lrb{
        \frac{P^\frac{1}{2}}{X^\frac{1}{2}}
        + 1
    }
    \frac{P^\frac{1}{2} Y^\frac{1}{2}}{X}
    \EE
    \\
    &\ll_\varepsilon  k \lrb{X^{\frac{1}{2}(1+\delta_1) - \frac{1}{2}} + X^0} X^{\frac{1}{2}(1+\delta_1) + \frac{1}{2}(1-\delta_2) - 1} \EE \\
    &\ll_\varepsilon  k X^{\frac{-\delta_2}{2}  + \delta_1} \EE\\
    &\ll_\varepsilon 
    k X^{\frac{-\delta_2}{2}  + \delta_1}
    \lrb{
        1 
        + 
        \frac{P^\frac{1}{10} X^\frac{3}{5}}{Y^\frac{9}{10}}
        +
        \frac{X}{Y^\frac{11}{9}}
        +
        \frac{X^\frac{1}{2}}{Y^\frac{11}{18}}
    }
    X^\varepsilon \\
    &\ll_\varepsilon  
    k X^{\frac{-\delta_2}{2}  + \delta_1}
    \lrb{
        X^0 + 
        X^{\frac{1}{10}(1+\delta_1) + \frac{3}{5} - \frac{9}{10}(1-\delta_2)} +
        X^{1-\frac{11}{9}(1-\delta_2)}
        +
        X^{\frac{1}{2} - \frac{11}{18} (1-\delta_2)}
    } X^{\varepsilon} \\
    &\ll_\varepsilon  
    k X^{\frac{-\delta_2}{2}  + \delta_1}
    \lrb{
        X^{\max \lrcb{
            0,\ \frac{-1}{5} + \frac{\delta_1}{10} + \frac{9\delta_2}{10},\ 
            \frac{-2}{9} + \frac{11\delta_2}{9},\ 
            \frac{-1}{9} + \frac{11\delta_2}{18}
        }}
    } X^{\varepsilon} \\
    &=
    k X^{\frac{-\delta_2}{2}  + \delta_1}
    \lrb{
        X^{\max \lrcb{
            0,\ 
            \frac{-2}{9} + \frac{11\delta_2}{9}
        }}
    } X^{\varepsilon} \\
    &= k X^{-\delta'+\varepsilon},
\end{align}
and the third and fourth big-$O$ error terms are bounded by
\begin{align}
    \frac{1}{X} + \frac{P}{X^\frac{1}{2} Y} + \frac{k}{Y} &\ll X^{-1} + X^{(1+\delta_1) - \frac{1}{2} - (1-\delta_2)} + kX^{-(1-\delta_2)} \\
    &\ll X^{-1} + X^{\frac{-1}{2}+\frac{1}{11}+\frac{4}{13}} + k X^{-1 + \frac{4}{13}}\\
    &\ll kX^{\frac{-1}{11}} \\
    &\ll k X^{-\delta'+\varepsilon}, \label{eqn:temp-third-and-fourth-big-O}
\end{align}
as desired. This completes the proof.
\end{proof}

Finally, we give an estimate for the average growth rate of $\dim S_k^\nw(N)$, which follows from a dimension formula of Martin \cite[Theorem 1]{Martin}.
\begin{lemma}
    \label{lem:dimension-estimate}
    Let 
    \begin{align}
        C_0 := \frac{1}{\zeta(2)} \prod_p \lrp{1- \frac{1}{p^2+p}}.
    \end{align}
    Then
    \begin{align}
        \sideset{}{^\square} \sum_{\substack{N \in [X, X+Y]}} \sum_{f \in H^\new(N,k)} 1
        &= \frac{k-1}{12} C_0 XY + O\lrp{kY^2 + kX^{\frac{3}{2}+\varepsilon}}.
    \end{align}
\end{lemma}
\begin{proof}
    The dimension formula \cite[Theorem 1]{Martin} implies that for squarefree $N$,
    \begin{align} \label{eqn:dim-formula-for-SknewN}
        \dim S_k^\nw(N) = \frac{k-1}{12} \phi(N) + O_\varepsilon(N^\varepsilon).
    \end{align}
    Moreover, we have from \cite[Lemma 3.21]{Zubrilina} that
    \begin{align}
        \sideset{}{^\square} \sum_{N\le Z} \phi(N) = \frac{C_0}{2} Z^2 + O_\varepsilon(Z^{\frac{3}{2}+\varepsilon}).
    \end{align}
    
    Then combining these two facts yields the estimate
    \begin{align}
        \sideset{}{^\square} \sum_{\substack{N \in [X, X+Y]}} \sum_{f \in H^\new(N,k)} 1
        &= \sideset{}{^\square} \sum_{\substack{N \in [X, X+Y]}} \lrb{
            \frac{k-1}{12} \phi(N) + O_\varepsilon(N^\varepsilon)
        }\\
        &= \frac{k-1}{12} \frac{C_0}{2} \lrb{(X+Y)^2 - X^2} + O_\varepsilon(k X^{\frac{3}{2}+\varepsilon}) + O_\varepsilon(YX^\varepsilon) \\
        &= \frac{k-1}{12} C_0 XY + O_\varepsilon(kY^2 + k X^{\frac{3}{2}+\varepsilon}),
    \end{align}
    as desired.
\end{proof}

\begin{remark}
    We note that \cite[Equation (22)]{Zubrilina} obtained a similar result to Lemma \ref{lem:dimension-estimate}. However, it seems that the statement of \cite[Equation (22)]{Zubrilina} omitted the $O_\varepsilon(kY^2)$ error term that arises from $\frac{k-1}{12} \frac{C_0}{2} \big[(X+Y)^2 - X^2\big]$.
\end{remark}

Combining Proposition \ref{prop:pre-main-result} and Lemma \ref{lem:dimension-estimate} then 
yields Theorem \ref{thm:main-result} with the same constant $\delta'$ as in Proposition \ref{prop:pre-main-result}: 
$
    \delta' = \frac{\delta_2}{2} - \delta_1 +\min\lrcb{0,\frac{2}{9}-\frac{11\delta_2}{9}}
$.

{
\renewcommand{\thetheorem}{\ref{thm:main-result}}
\addtocounter{theorem}{-1}
\begin{theorem}
Fix an even integer $k\ge 2$. Let $P$, $X$, and $Y$ be parameters going to infinity with $X,Y\in \mathbb{R}_{+}$ and $P$ prime; assume further that $Y=(1+o(1))X^{1-\delta_2},$ and $P\ll X^{1+\delta_1}$ for some $\delta_1,\delta_2$ with $0<\delta_1<\frac{1}{11}$, $2\delta_1<\delta_2<\frac{1}{13} (4-18\delta_1).$ Let $y := \frac{P}{X}$. Then 
\begin{align}
    &\frac{
    \sum_{N\in[X,X+Y]}^\square \sum_{f\in H^{\rm new}(N,k)}\sqrt{P}\lambda_f(P)\varepsilon(f)
    }{
    \sum_{N\in[X,X+Y]}^\square \sum_{f\in H^{\rm new}(N,k)}1 
    }  \\
    &=\frac{\alpha(-1)^{\frac{k}{2}-1}}{k-1}\sum_{1\leq r\leq 2\sqrt{y}}U_{k-2}\lrp{\frac{r}{2\sqrt{y}}}\nu(r)\sqrt{4y-r^2} +\frac{\beta}{k-1}\sqrt{y}-\gamma \delta_{k=2}y 
    \label{eqn:temp_main-term-of-main-thm}
    \\
    &\quad\ +O_\varepsilon\lrp{X^{-\delta'+\varepsilon}+\frac{1}{P}}
\end{align}
where 
$
    \delta' := \frac{\delta_2}{2} - \delta_1 +\min\lrcb{0,\frac{2}{9}-\frac{11\delta_2}{9}} > 0
$,
and 
\begin{align}
\alpha= 2\pi \prod_{p}\frac{1-p-2p^2+p^4}{p^4-2p^2+p}, \quad\ \ \beta=2\pi\prod_{p}\frac{-1+p^2+p^3}{p(-1+p+p^2)}, \quad \ \ 
\gamma=12\prod_{p}\frac{p(1+p)}{-1+p+p^2}.
\end{align}
\end{theorem}
}

\begin{proof}
    Lemma \ref{lem:dimension-estimate} implies that
    \begin{align}
        \frac{1}{
            \sideset{}{^\square} \sum_{\substack{N \in [X, X+Y]}} \sum_{f \in H^\new(N,k)} 1
        }
        &= \frac{1}{
            \frac{k-1}{12} C_0 XY 
        }
        \lrb{
            1 + O_\varepsilon\lrp{\frac{Y}{X} + \frac{X^{\frac{1}{2}+\varepsilon}}{Y}}
        }.
    \end{align}
    \noindent
    Thus by Proposition \ref{prop:pre-main-result}, we have that 
    \begin{align}
    &\frac{
    \sum_{N\in[X,X+Y]}^\square \sum_{f\in H^{\rm new}(N,k)}\sqrt{P}\lambda_f(P)\varepsilon(f)
    }{
    \sum_{N\in[X,X+Y]}^\square \sum_{f\in H^{\rm new}(N,k)}1 
    }  \\
    &= \frac{1}{
            \frac{k-1}{12} C_0  
        }
        \lrb{
            1 + O_\varepsilon\lrp{\frac{Y}{X} + \frac{X^{\frac{1}{2}+\varepsilon}}{Y}}
        } \times
    \\
    &
    \lrb{
        (-1)^{\frac{k}{2}-1} B_0\!\!\!\sum_{1 \le r \le 2\sqrt{y}}   U_{k-2}\lrp{\frac{r}{2\sqrt{y}}}\nu(r)\sqrt{4y-r^2} 
        + A_0\sqrt{y} - \frac{\delta_{k=2}}{\zeta(2)} y  + O_\varepsilon\lrp{kX^{-\delta'+\varepsilon}+\frac{k}{P}}
    } \\
    &= \frac{\alpha(-1)^{\frac{k}{2}-1}}{k-1}\sum_{1\leq r\leq 2\sqrt{y}}U_{k-2}\lrp{\frac{r}{2\sqrt{y}}}\nu(r)\sqrt{4y-r^2} +\frac{\beta}{k-1}\sqrt{y}-\gamma \delta_{k=2}y \\
    &\quad \ + O_\varepsilon\lrp{
        (y+\sqrt y) \cdot
        \lrb{
            \frac{Y}{X} + \frac{X^{\frac{1}{2}+\varepsilon}}{Y} 
        }
        +
        1 \cdot 
        \lrb{
            X^{-\delta'+\varepsilon} + \frac{1}{P}
        }
    }, \label{eqn:temp-main-thm-error-term}
    \\
    &\qquad \text{\Big(since the main term \eqref{eqn:temp_main-term-of-main-thm} is $O\lrp{k\big(y+\sqrt y\big)}$\Big)}
\end{align}
where
\begin{align}
    \alpha &= \frac{12}{C_0} B_0 = 2\pi 
    \prod_p 
    \frac{p^2+p}{p^2+p-1} \cdot \frac{p^2-1}{p^2} \cdot \frac{p^4 - 2p^2 - p + 1}{(p^2-1)^2}
    = 2\pi \prod_{p}\frac{1-p-2p^2+p^4}{p^4-2p^2+p},
    \\
    \beta  &= \frac{12}{C_0} A_0
    = 2\pi 
    \prod_p 
    \frac{p^2+p}{p^2+p-1} \cdot \frac{p^2-1}{p^2}
    \cdot \frac{p^3+p^2-1}{(p+1)^2(p-1)}
    = 2\pi\prod_{p}\frac{-1+p^2+p^3}{p(-1+p+p^2)},
    \\
    \gamma &= \frac{12}{C_0} \frac{1}{\zeta(2)} = 12 \prod_p \frac{p^2+p}{p^2+p-1} = 12 \prod_p \frac{p(1+p)}{-1+p+p^2}.
\end{align}
Moreover, the big-$O$ error term \eqref{eqn:temp-main-thm-error-term} is bounded by
\begin{align}
    & (y+\sqrt y) \cdot
        \lrb{
            \frac{Y}{X} + \frac{X^{\frac{1}{2}+\varepsilon}}{Y} 
        }
        +
        1 \cdot 
        \lrb{
            X^{-\delta'+\varepsilon} + \frac{1}{P}
        } \\
    &\ll_\varepsilon  
        \lrp{
            X^{(1+\delta_1)-1} + X^{\big(\frac{1}{2} + \frac{\delta_1}{2}\big) - \frac{1}{2}}
        }
        \lrb{
            X^{-\delta_2}
            + X^{\frac{1}{2} - (1 - \delta_2)+\varepsilon}
        }
        +
        \lrb{X^{-\delta'+\varepsilon} + \frac{1}{P}} \\
    &\ll_\varepsilon 
        X^{\delta_1} 
        \lrb{
            X^{-\delta_2}
            + X^{\frac{-1}{2} + \frac{4}{13} +\varepsilon}
        }
        +
        \lrb{X^{-\delta'+\varepsilon} + \frac{1}{P}} \\
    &\ll_\varepsilon 
    X^{\delta_1 - \frac{\delta_2}{2}} + \lrb{X^{-\delta'+\varepsilon} + \frac{1}{P}} \\
    &\ll_\varepsilon 
    X^{-\delta'+\varepsilon} + \frac{1}{P},
\end{align}
completing the proof.
\end{proof}

\section{Proof of Proposition \ref{prop:S-estimate}} 
\label{sec:proof-of-prop-S-estimate}

In this section, we prove Proposition \ref{prop:S-estimate}, which will be used to prove Theorem \ref{thm:second-main-result}.

\begin{remark}
    The proof strategy for Theorem \ref{thm:second-main-result} laid out in \cite{Zubrilina} relied on a certain bound $R(T) \ll_\varepsilon  T^{-2+\varepsilon}$ given in \cite[Lemma 7.3]{Zubrilina}. However, as noted in the introduction, the argument given in \cite{Zubrilina} seems to not actually yield this bound. For this reason, we will proceed by a different strategy.
\end{remark}

Define the arithmetic functions
\begin{align*}
    Q(n):=\mu^2(n) \prod_{p\mid n}\frac{p^2}{p^4-2p^2-p+1} 
    \qquad\text{and}\qquad
    b(n) := n^2 Q(n).
\end{align*}

We first prove the following estimate for the partial sums of $b(n)$.
\begin{lemma} \label{lem:B-approx}
    For $x \ge 1$, we have
    \begin{align}
        B(x) := \sum_{n \le x} b(n) = \Delta x + O\lrp{\sqrt x},
    \end{align}
    where 
    \begin{align}
        \Delta:=\frac{2\pi}{\alpha\zeta(2)} =\frac{1}{\zeta(2)}\prod_p\left(1+\frac{2p-1}{p^4-2p^2-p+1}\right).
    \end{align}
\end{lemma}
\begin{proof}
    Observe that
    \begin{align*} b(n)=\mu(n)^2\prod_{p\mid n}c_p,\qquad \text{where } c_p:=\frac{p^4}{p^4-2p^2-p+1}.
    \end{align*}
    Note that $c_p-1\ll\frac{1}{p^2}$. 
    Define the multiplicative function $g$ by
    \[
        g(p^e)=(-1)^{e-1}(c_p-1)
        \qquad \text{for } e\ge1.
    \]
    Then we have that $b = \mu^2 * g$ 
    because of the formal power series identity
    \begin{align*}
        \sum_{e\ge0}b(p^e)x^e
        &=
        1+c_px \\
        &=
        (1+x)
        \left(
            1+\frac{(c_p-1)x}{1+x}
        \right) \\
        &=
        \left(\sum_{e\ge0}\mu^2(p^e)x^e\right)
        \left(\sum_{e\ge0}g(p^e)x^e\right).
    \end{align*}

    Moreover, we have that
    \begin{align} \label{eq:h(n)-convergence}
        \sum_{n\ge1}\frac{|g(n)|}{\sqrt{n}}
        =
        \prod_p
        \sum_{e\ge0} \frac{|g(p^e)|}{p^{\frac{e}{2}}}
        =\prod_p\left(1+(c_p-1)\frac{p^{\frac{-1}{2}}}{1-p^{\frac{-1}{2}}}\right)
        < \infty,
    \end{align}
    since each local factor is $1+O\big(p^{\frac{-5}{2}}\big)$.

    This then implies that
    \begin{align}
        B(x)
        \,\,&\!\!:=\sum_{n\le x}b(n) \\
        &= \sum_{n \le x} \sum_{d \mid n} g(d) \mu^2\lrp{\frac{n}{d}} 
        \\
        &=\sum_{d\le x}g(d)\sum_{m\le \frac{x}{d}}\mu^2(m)\\
        &= \sum_{d\le x}g(d) \lrb{
            \frac{1}{\zeta(2)} \frac{x}{d} + O\lrp{\sqrt{\frac{x}{d}}}
        } 
        \\
        &\qquad \text{(since $\textstyle \sum_{n\le X} \mu^2(n) = \frac{X}{\zeta(2)} + O(\sqrt X)$ by \cite[Exercise 4.4.10]{murty})}
        \\
        &= \frac{x}{\zeta(2)}\sum_{d\le x}\frac{g(d)}{d}+O\left(\sqrt{x}\sum_{d\le x}\frac{|g(d)|}{\sqrt{d}}\right) \\
        &= \lrb{
            \frac{x}{\zeta(2)} \sum_{d \ge 1} \frac{g(d)}{d} - \frac{x}{\zeta(2)} \sum_{d > x} \frac{g(d)}{d}
        }
        + O(\sqrt{x}) 
        \qquad\qquad \text{(by \eqref{eq:h(n)-convergence})}
        \\
        &= \Delta x+O(\sqrt{x}),
        \qquad \text{\big(since 
        $\textstyle
        \lrabs{x
        \sum_{d>x}\frac{g(d)}{d}
        }
        \le
        \sqrt{x}
        \sum_{d>x}\frac{|g(d)|}{\sqrt d}
        \ll \sqrt{x}
        $
        by \eqref{eq:h(n)-convergence}\big)}
    \end{align}    
where 
\begin{align*}
    \Delta=&\frac{1}{\zeta(2)}\sum_{d\ge1}\frac{g(d)}{d}\\
    =&\frac{1}{\zeta(2)}\prod_p\left(1+(c_p-1)\frac{p^{-1}}{1+p^{-1}}\right)\\
    =&\frac{1}{\zeta(2)}\prod_p\left(1+\frac{2p-1}{p^4-2p^2-p+1}\right) \\
    =& \frac{2\pi}{\alpha\zeta(2)},
\end{align*}
where $\alpha$ is as defined in the statement of Theorem \ref{thm:main-result}. This completes the proof.
\end{proof}

Next, we cite the following lemma from \cite{Zubrilina}. In the following, $J_{k-1}(x)$ denotes the Bessel function of the
first kind of order $k-1$.
\begin{lemma}[{First two displayed equations after the proof of \cite[Lemma 7.3]{Zubrilina}}] \label{lem:properties-of-J}
    Fix an even integer $k \ge 6$. Then there exists a finite linear combination $J$ of Bessel functions of orders at least $4$ such that
    \begin{align} \label{eq:J-definition}
        \frac{d}{dx}
        \left[\frac{J(x)}{x^4}\right]
        =
        \frac{J_{k-1}(x)}{x^4}.
    \end{align}
    Moreover, this function $J(x)$ satisfies
    \begin{align} \label{eq:J-integral}
        \int_0^\infty \frac{J(x)}{x}\, dx
        =-\frac14.
    \end{align}
\end{lemma}

We are now ready to prove Proposition \ref{prop:S-estimate}. In the following, $J$ denotes the function from Lemma \ref{lem:properties-of-J}.
\begin{proposition} \label{prop:S-estimate}
    For $a\ge 1$, we have that
    \begin{align}
        S(a):=\sum_{d\ge 1}dQ(d)J\lrp{\frac{a}{d}} = -\frac{\Delta}{4} + O_k\lrp{a^{-\frac{1}{6}}} 
    \end{align}
\end{proposition}
\begin{proof}
First, define
\[
    F(x):=\frac{1}{x}J\lrp{\frac{1}{x}}.
\]
Since $J$ is a finite linear combination of Bessel functions of indices at least $4$,
we have the standard Bessel function estimates from \cite[Sections 10.14 and 10.17]{NIST-DLMF}:
\begin{align}
    J(t)
    &\ll_k
    \min\lrcb{t^4, t^{\frac{-1}{2}}}
    &
    \text{uniformly over $t \in (0,\infty)$},
    \\
    % \begin{cases}
    %     t^4, & 0<t\le1,\\
    %     t^{\frac{-1}{2}}, & t\ge1,
    % \end{cases} \\
    J'(t)
    &\ll_k
    % \begin{cases}
    %     t^3, & 0<t\le1,\\
    %     t^{\frac{-1}{2}}, & t\ge1.
    % \end{cases}
    t^{\frac{-1}{2}}
    &
    \text{uniformly over $t \in (0,\infty)$}.
\end{align}
This then immediately implies that
\begin{align} 
    F(x)
    &\ll_k
    \min\lrcb{x^{-5}, x^{\frac{-1}{2}}}
    &
    \text{uniformly over $x \in (0,\infty)$},
    % \begin{cases}
    %     x^{\frac{-1}{2}}, & 0<x\le1,\\
    %     x^{-5},   & x\ge1,
    % \end{cases}
    \label{eq:F-bounds}
    \\
    F'(x)
    &=
    -\frac{J\lrp{\frac{1}{x}}}{x^2}
    -
    \frac{J'\lrp{\frac{1}{x}}}{x^3}
    \ll_k
    x^{\frac{-5}{2}}
    &
    \text{uniformly over $x \in (0,\infty)$}.
    % \min\lrcb{x^{-6}, x^{\frac{-5}{2}}}
    % \begin{cases}
    %     x^{\frac{-5}{2}}, & 0<x\le1,\\
    %     x^{-6},   & x\ge1.
    % \end{cases}
    \label{eq:F'-bounds}
\end{align}
Also, note that
\begin{align}
    \label{eqn:b(d)<<1}
    b(d)
    =
    \mu^2(d)
    \prod_{p\mid d}c_p
    \le
    \prod_p \lrb{1 + O\big(p^{-2}\big)}
    \ll1.
\end{align}

Now, we prove the desired result. Since $b(d)=d^2Q(d)$, we have
\begin{align} \label{eq:S-b-expression}
S(a)
&=
\sum_{d\ge1} \frac{b(d)}{d} J\lrp{\frac{a}{d}}
\\
&=
\frac1a\sum_{d\ge 1}b(d)F\lrp{\frac{d}{a}} \\
&= 
\frac1a\sum_{d\le a^{\frac{2}{3}}}b(d)F\lrp{\frac{d}{a}}  + \frac1a\sum_{d> a^{\frac{2}{3}}}b(d)F\lrp{\frac{d}{a}}.
\end{align}
We deal with these two summations separately.

First, we have that
\begin{align} \label{eq:small-d}
\lrabs{\frac1a\sum_{d\le a^{\frac{2}{3}}}b(d) F\Big(\frac{d}{a}\Big)}
&\ll_k
\frac1{a}\sum_{d\le a^{\frac{2}{3}}}\lrp{\frac{d}{a}}^{\frac{-1}{2}} 
\qquad \text{(by \eqref{eqn:b(d)<<1} and \eqref{eq:F-bounds})}
\\
&\ll_k
\frac{a^{\frac{1}{2}}}{a} \cdot 
\lrp{a^{\frac{2}{3}}}^{\frac{1}{2}} 
\\
&=
a^{\frac{-1}{6}},
\end{align}
which is bounded by the claimed error term.

Second, we have by partial summation from $a^\frac{2}{3}$ to $\infty$ that
\begin{align}
&\frac1a\sum_{d>a^{\frac{2}{3}}}b(d)F\lrp{\frac{d}{a}} \\
&=
\frac{1}{a}
\lrb{
    \lim_{x\to\infty} B(x) F\lrp{\frac{x}{a}}
    -B(a^{\frac{2}{3}}) F\bigg(\frac{a^\frac{2}{3}}{a}\bigg)
    -\frac1{a}\int_{a^{\frac{2}{3}}}^\infty B(t)F'\lrp{\frac{t}{a}}\,dt
}
\\
&=
\frac{1}{a}
\lrb{
    -\lrb{\Delta a^{\frac{2}{3}} + E(a^{\frac{2}{3}}) } F\big(a^{\frac{-1}{3}}\big)
    -\frac1{a}\int_{a^{\frac{2}{3}}}^\infty \big[\Delta t + E(t)\big] F'\lrp{\frac{t}{a}}\,dt
}
\\
& \qquad \text{(where $E(t) := B(t) - \Delta t \ll \sqrt t$ by Lemma \ref{lem:B-approx})} \\
&= - \Delta a^{\frac{-1}{3}} F\big(a^{\frac{-1}{3}}\big) - \frac{\Delta}{a^2} \int_{a^{\frac{2}{3}}}^\infty t F'\lrp{\frac{t}{a}}\,dt + \lrb{
    -\frac{E(a^{\frac{2}{3}})}{a} F\big(a^{\frac{-1}{3}}\big)
    - \frac{1}{a^2} \int_{a^{\frac{2}{3}}}^\infty E(t) F'\lrp{\frac{t}{a}}\,dt
} 
\\
&=:
- \Delta a^{\frac{-1}{3}} F\big(a^{\frac{-1}{3}}\big)
-\Delta\int_{a^{\frac{-1}{3}}}^\infty x F'(x)\,dx
+ \hat E(a) 
\\
&=
\Delta\int_{a^{\frac{-1}{3}}}^\infty F(x)\,dx
+ \hat E(a)
\qquad\qquad\quad\ \ \,
\text{\big(by IBP, since $\lim_{x \to \infty} xF(x) = 0$ by \eqref{eq:F-bounds}\big)}
\\
&= \Delta\int_0^\infty F(x)\,dx
+O_k\big(a^{\frac{-1}{6}}\big) +\hat E(a)
\quad\,
\text{\big(since 
  $\textstyle
    \int_0^{a^{\frac{-1}{3}}} \!\lvert F(x)\rvert dx
    \ll
    \int_0^{a^{\frac{-1}{3}}} \!\!x^{\frac{-1}{2}}dx
    =
    2a^{\frac{-1}{6}}
  $\big)} 
\\
&=
\Delta\int_0^\infty \frac{J(u)}{u}\,du
+O_k\big(a^{\frac{-1}{6}}\big) + \hat E(a)
\quad
\text{\big(letting $u=\tfrac{1}{x}$\big)}
\\
&=
-\frac{\Delta}{4}
+O_k\big(a^{\frac{-1}{6}}\big) + \hat E(a),
\qquad\qquad\quad\ 
\text{(by \eqref{eq:J-integral})}
\end{align}
which gives the main term of the desired result.

Additionally, observe that
\begin{align}
\lrabs{\hat E(a)}
&= 
\lrabs{
    -\frac{E(a^{\frac{2}{3}})}{a} F\big(a^{\frac{-1}{3}}\big)
    - \frac{1}{a^2} \int_{a^{\frac{2}{3}}}^\infty E(t) F'\lrp{\frac{t}{a}}\,dt
} \\
&\ll_k
\frac{|E(a^{\frac{2}{3}})|}{a} |F\big(a^{\frac{-1}{3}}\big)|
+
\frac{1}{a} 
\int_{a^{\frac{-1}{3}}}^\infty
\lrabs{
    E\lrp{ax}
}
|F'(x)|\,dx
\\
&\ll_k
\frac{\sqrt{a^\frac{2}{3}}}{a} \lrp{a^\frac{-1}{3}}^\frac{-1}{2}
+
\frac{1}{a}
\int_{a^{\frac{-1}{3}}}^\infty
\sqrt{ax}
\cdot
x^\frac{-5}{2}\,dx
\qquad
\text{(by \eqref{eq:F-bounds} and \eqref{eq:F'-bounds})}
\\
&=
a^{\frac{-1}{2}}
+
a^{\frac{-1}{2}}
\left(a^\frac{-1}{3}\right)^{-1}
\\
&\ll_k
a^{\frac{-1}{6}},
\end{align}
which is bounded by the claimed error term. This
completes the proof.
\end{proof}

\section{Proof of Theorem \ref{thm:second-main-result}}
\label{sec:proof-of-second-main-result}

For fixed even integers $k\ge 2$, we begin by defining the weight $k$ murmuration density as
\begin{align} \label{eq:Mk-bessel}
    \M_k(y)
    :=
    \alpha\sqrt{y}
    \sum_{s\ge1}\sum_{d\ge1}
    \frac{Q(d)}{s}
        J_{k-1}\!\left(\frac{4\pi s\sqrt{y}}{d}\right)
    \qquad \text{for $y>0$}.
\end{align}
We note that \cite{Zubrilina} actually defined $\M_k(y)$ in a slightly different way, then proves that it is equal to this double summation in \cite[Theorem 3]{Zubrilina}. However, just defining $\M_k(y)$ as this summation will suffice for our purposes.

Moreover, we note that the standard Bessel function bound $J_{k-1}(x) \ll_k x^{\frac{-1}{2}}$ and the bound $Q(d) \ll \frac{1}{d^2}$ immediately imply the following facts about $\M_k(y)$.
\begin{lemma} \label{lem:Mk-convergence}
    The double summation
    \eqref{eq:Mk-bessel} converges absolutely and locally uniformly for
    $y\in(0,\infty)$. Moreover,
    \begin{align} \label{eq:Mk-at-zero}
        \M_k(y)\ll_k y^{\frac{1}{4}}
        \qquad \text{uniformly over $y \in (0,\infty)$},
    \end{align}
    so $\M_k$ extends continuously to $y=0$ by setting $\M_k(0)=0$.
\end{lemma}
\begin{proof}
Applying the bounds
$J_{k-1}(x)\ll_k x^{\frac{-1}{2}}$ and $Q(d)=\frac{b(d)}{d^2}\ll d^{-2}$, we obtain that
\begin{align}
\sum_{s\ge1}
\sum_{d\ge 1}
\left|
\alpha\sqrt y\,\frac{Q(d)}s
J_{k-1}\left(\frac{4\pi s\sqrt y}{d}\right)
\right|
&\ll_k y^{\frac{1}{4}}
\left(\sum_{d\ge1}Q(d)\sqrt d\right)
\left(\sum_{s\ge1}s^{\frac{-3}{2}}\right) \\
&\ll_k y^{\frac{1}{4}} \qquad \text{uniformly over $y \in (0,\infty)$.}
\end{align}
This proves the desired result.
\end{proof}

We are finally ready to prove Theorem \ref{thm:second-main-result}. 
{
\renewcommand{\thetheorem}{\ref{thm:second-main-result}}
\addtocounter{theorem}{-1}
\begin{theorem} 
    For $c>1$ and $k\ge6$ even, define the smoothing of $\M_k(y)$ as 
    \begin{align}
        \M_k^{c}(y)
        :=
        \frac{\displaystyle\int_1^c \M_k\lrp{\frac{y}{u}}u\, du}
        {\displaystyle\int_1^c u\, d u}.
    \end{align}
    Then $\M_k^{c}$ is continuous on $[0,\infty)$ with $\M_k^{c}(0)=0$ and as $y\to\infty$,
    \begin{align}
        \M_k^{c}(y)=\frac{1}{2}+O_{k,c}\lrp{y^{\frac{-1}{12}}}.
    \end{align}
\end{theorem}
}
\begin{proof}
Continuity follows immediately from Lemma \ref{lem:Mk-convergence}. So it remains to determine the behavior of $\M^{c}_k(y)$ as $y\to\infty$.
We have
\begin{align}
&\int_1^c
\M_k\lrp{\frac{y}{u}}u\,du \\
&=
\int_1^c
\alpha
\sqrt{\frac{y}{u}}
\sum_{s\ge1} \sum_{d\ge1}
\frac{Q(d)}{s}
J_{k-1}\lrp{
\frac{4\pi s}{d}
\sqrt{\frac{y}{u}}
}
u\,du
\qquad \text{(by \eqref{eq:Mk-bessel})}
\\
&=
\alpha \sqrt{y}
\sum_{s\ge1} \sum_{d\ge1}
\frac{Q(d)}{s}
\int_1^c
\sqrt{u}\,
J_{k-1}\lrp{\frac{4\pi s\sqrt{y}}{d\sqrt{u}}}\,du
\\
&=
2\alpha \sqrt y
\sum_{s\ge1} \sum_{d\ge1}
\frac{Q(d)}{s}
\int_{\frac{1}{\sqrt c}}^{1}
\frac{1}{w^4}
J_{k-1}\lrp{\frac{4\pi s \sqrt{y}}{d} w}\,dw
\qquad\quad
\text{\big(letting $w=\tfrac{1}{\sqrt u}$\big)}
\\
&=
2 \alpha \sqrt y
\sum_{s\ge1} \sum_{d\ge1}
\frac{Q(d)}{s} \lrp{\frac{4\pi s \sqrt{y}}{d}}^3
\int_{\frac{4\pi s\sqrt y}{d\sqrt c}}^{\frac{4\pi s\sqrt y}{d}}
\frac{J_{k-1}(t)}{t^4}\,dt
\qquad\ \,
\text{\big(letting $t=\tfrac{4\pi s\sqrt y}{d} w$\big)}
\\
&=
\frac{\alpha}{2\pi}
\sum_{s\ge1} \sum_{d\ge1}
\frac{d Q(d)}{s^2} \lrp{\frac{4\pi s \sqrt{y}}{d}}^4
\lrb{
\frac{
J\lrp{\frac{4\pi s\sqrt y}{d}}
}{
\lrp{\frac{4\pi s\sqrt y}{d}}^4
}
-
\frac{
J\lrp{\frac{4\pi s\sqrt y}{d\sqrt c}}
}{
\lrp{\frac{4\pi s\sqrt y}{d\sqrt c}}^4
}
}
\qquad\quad\!\!
\text{(by \eqref{eq:J-definition})}
\\
&=
\frac{\alpha}{2\pi}
\sum_{s\ge1} \sum_{d\ge1}
\frac{dQ(d)}{s^2}
\lrb{
J\lrp{\frac{4\pi s\sqrt y}{d}}
-
c^2
J\lrp{\frac{4\pi s\sqrt y}{d\sqrt c}}
}
\\
&=
\frac{\alpha}{2\pi}
\sum_{s\ge1}\frac{1}{s^2}
\lrb{
    S\lrp{4\pi s\sqrt y}
    -
    c^2S\lrp{\frac{4\pi s \sqrt y}{\sqrt c}}
} \\
&=
\frac{\alpha}{2\pi}
\sum_{s\ge1}\frac{1}{s^2}
\lrb{
    \frac{-\Delta}{4} 
    - c^2 \frac{-\Delta}{4}
    +O_{k,c}\lrp{\lrp{s \sqrt y}^{\frac{-1}{6}}}
} 
\qquad \text{(by Proposition \ref{prop:S-estimate})}
\\
&=
\frac{\alpha}{2\pi} \cdot \zeta(2) \cdot \frac{(c^2-1)\Delta}{4} + O_{k,c}\lrp{y^\frac{-1}{12}} \\
&= 
\frac{c^2-1}{4}
+O_{k,c}\lrp{y^{\frac{-1}{12}}}. \qquad \text{(since $\textstyle \Delta = \frac{2\pi}{\alpha \zeta(2)}$)}
\end{align}

This then yields
\begin{align}
    \M_k^{\,c}(y)
    =
    \frac{
        \displaystyle
        \int_1^c
        \M_k\lrp{\frac{y}{u}}u\,du
    }{
        \displaystyle
        \int_1^c u\,du
    }
    =
    \frac{
        \displaystyle
        \frac{c^2-1}{4}
        +O_{k,c}\lrp{y^{\frac{-1}{12}}}
    }{
        \displaystyle
        \frac{c^2-1}{2}
    }
    =
    \frac12
    +
    O_{k,c}\lrp{y^{\frac{-1}{12}}},
\end{align}
as desired.
\end{proof}

% In closing, we briefly note two corrections to the statements of \cite[Theorem 2]{Zubrilina} and \cite[Theorem 3]{Zubrilina}. 
% In the statement of \cite[Theorem 2]{Zubrilina}, it seems that the bound $P \ll X^{\frac{6}{5}}$ in the statement of \cite[Theorem 2]{Zubrilina} was intended to be $P \ll X^{\frac{12}{11}-\varepsilon}$.
% Additionally, in the statement of \cite[Theorem 3]{Zubrilina}, it seems that the leading constant should be $(-1)^{\frac{k}{2}-1} \frac{1}{\pi \sqrt{2}} \alpha$ instead of $(-1)^{\frac{k}{2}-1} \sqrt{\frac{2}{\pi}} \alpha$. 

\section*{Acknowledgments}
The authors would like to thank Nina Zubrilina for her work on murmurations, which inspired this study. This research was supported by NSF grant DMS-2349174. Hui Xue is supported by Simons Foundation grant MPS-TSM-00007911.

\bibliographystyle{plain}
\bibliography{bibliography.bib}

\end{document}